\documentclass[11pt]{amsart}
\usepackage{amsmath,amssymb}
\usepackage{verbatim}
\usepackage{color}

\usepackage{graphicx}
\newcommand{\Z}{{\mathbb Z}} 
\newcommand{\Q}{{\mathbb Q}}
\newcommand{\R}{{\mathbb R}}
\newcommand{\C}{{\mathbb C}}

\newcommand{\Gal}{\mathrm{Gal}}

\newcommand{\Hom}{\mathrm{Hom}}
\newcommand{\GL}{\mathrm{GL}}
\newcommand{\PGL}{\mathrm{PGL}}
\renewcommand{\mod}{\;\operatorname{mod}}
\newcommand{\ord}{\operatorname{ord}}

\newcommand{\ff}{\frak f}

\newcommand{\fa}{\mathfrak a}

\newcommand{\Ind}{\operatorname{Ind}} 
\newcommand{\Frob}{\operatorname{Frob}} 
 \newcommand{\Res}{\operatorname{Res}} 
 
\begin{document}

\newtheorem{thm}{Theorem}[section]
\newtheorem{theorem}{Theorem}[section]
\newtheorem{lem}[thm]{Lemma}
\newtheorem{lemma}[thm]{Lemma}
\newtheorem{prop}[thm]{Proposition}
\newtheorem{proposition}[thm]{Proposition}
\newtheorem{cor}[thm]{Corollary}
\newtheorem{defn}[thm]{Definition}
\newtheorem*{remark}{Remark}
\newtheorem{conj}[thm]{Conjecture}

\title{Pair arithmetical equivalence for quadratic fields}





\author{Wen-Ching Winnie Li}
\address{Department of Mathematics, Pennsylvania State University,
University Park, PA 16802, USA}
\email{wli@math.psu.edu}

\author{Zeev Rudnick}
\address{School of Mathematics, Tel Aviv University,
Tel Aviv 69978, Israel}
\email{rudnick@tauex.tau.ac.il}

\thanks{The research of Li is partially supported by Simons Foundation grant \# 355798.  Z.R received funding from the European Research Council (ERC) under the European Union's  Horizon 2020 research and innovation programme  (Grant agreement No.    786758).}

\date{\today}

\begin{abstract}
Given two nonisomorphic number fields $K$ and $M$, and  finite order Hecke characters $\chi$ of $K$ and $\eta$ of $M$ respectively, we say that  the     pairs $(\chi, K)$ and $(\eta, M)$ are   arithmetically equivalent  if the associated   L-functions coincide: 
$$
L(s, \chi, K) = L(s, \eta, M) .
$$
When the characters are trivial, this reduces to the question of fields with the same Dedekind zeta function, investigated by Gassmann in 1926, who found such fields of degree 180, and by Perlis (1977) and others, who showed that there are no nonisomorphic fields of degree less than $7$. 
We construct infinitely many such pairs  where  the fields are quadratic. This gives dihedral automorphic forms induced from characters of  different quadratic  fields. 
We also give a classification of such characters of order $2$  for the quadratic fields of our examples, all with odd class number.

\end{abstract}


\maketitle

\section{Introduction}
\subsection{Arithmetic equivalence of fields}
Two number fields $K$ and $L$ are {\em arithmetically equivalent} if their Dedekind zeta functions coincide: $\zeta_K(s) = \zeta_L(s)$.  
A field is {\em arithmetically solitary} if it is isomorphic to any field with the same Dedekind zeta function. Examples are normal extensions of the rationals. 
The first non-solitary fields were found in 1926 by Gassmann \cite{Gassmann}, who discovered a pair of non-isomorphic fields of degree $180$ which are arithmetically equivalent. Perlis \cite{Perlis} showed that all the fields $K$ with $[K:\Q]\leq 6$ are arithmetically solitary, and constructed a non-solitary field of degree $7$.

A variant for Artin L-functions was investigated by Kl\"{u}ners and Nicolae \cite{KN}. For $j=1,2$ let $K_j/\Q$  
be a finite Galois extension, and let $\chi_j$ be a faithful character of the Galois group $G_j = \Gal(K_j/\Q)$. If the corresponding Artin L-functions coincide: 
$L(s,\chi_1,K_1/\Q)=L(s,\chi_2,K_2/\Q)$,  
then $K_1=K_2$ and $\chi_1=\chi_2$. They also showed that if the base field is not the rationals, this need not be true. 
For other variations on this theme, see \cite{Cor}, \cite{CdSLMS}, \cite{Prasad}, \cite{MS2019}, \cite{Sutherland}.

\subsection{Arithmetic pair equivalence}

In this paper we consider a different variant of field arithmetical equivalence, which we call arithmetic pair equivalence. 
For a number field $K$ denote by $G_K$ the absolute Galois group $\Gal(\bar \Q/K)$.  
Given two nonisomorphic number fields $K$ and $M$, let $\chi$ and $\eta$ be two finite order characters of $G_K$ and $G_M$ respectively.  Two pairs $(\chi, K)$ and $(\eta, M)$ are called {\em arithmetically equivalent}  if the associated  L-functions coincide: 
$$
L(s, \chi, K) = L(s, \eta, M) .
$$

 An immediate consequence of pair arithmetical equivalence is that the two fields $K$ and $M$ must have the same degree over $\Q$. Clearly, pair arithmetical equivalence reduces to field arithmetical equivalence when both characters are trivial.

The problem that we study here is the existence of such pairs, when the base field $K$ is a quadratic extension of the rationals. 
More precisely, given a quadratic extension $K$ of $\Q$, we wish to find a nontrivial  finite order character $\chi$ of $G_K$ so that the pair $(\chi, K)$ is arithmetically equivalent to another pair $(\eta, M)$.

For instance we take the imaginary quadratic field $K=\Q(\sqrt{-1})$ and the real quadratic field $M=\Q(\sqrt{q})$ where $q$ is any prime satisfying  $q=1\bmod 8$, and construct quadratic characters $\chi$ of $K$ and $\eta$ of $M$ such that $L(s,\chi,K) = L(s,\eta,M)$ 
(see \S~\ref{sec:holomorphic forms}).

\subsection{Connection with dihedral modular forms} 
 
We now connect with the theory of automorphic forms.  
Recall that  by class field theory, finite order characters of the Galois group $G_K$ may be identified with finite order characters of the idele class group of $K$ (Hecke characters),    and we shall freely identify the two.   
For $K$ quadratic over $\Q$, and a Hecke  character  $\chi$  of $K$, 
there is a unique normalized automorphic Hecke-eigenform $g_\chi$ of $GL(2)$ over $\Q$, which is cuspidal if $\chi$ is not self-conjugate, with associated L-function $L(s, g_\chi) = L(s, \chi, K)$. It corresponds to the two-dimensional dihedral representation $\rho_\chi:=\Ind_{G_K}^{G_\Q} \chi$ of $G_\Q$.  
We call $\rho_\chi$ odd if it has eigenvalues $\pm 1$ at the complex conjugation $c$ in $G_\Q$, otherwise it is called even, in which case $\rho_\chi(c) = \pm Id$. The newform $g_\chi$ is holomorphic of weight one if $\rho_\chi$ is odd, and is a Maass form with Laplacian eigenvalue $1/4$ if $\rho_\chi$ is even. 

In terms of Fourier expansions, $g_\chi$ is given as follows: 
Assume $\chi$ is not self-conjugate. For $\rho_\chi$ odd, the holomorphic weight one cusp form $g_\chi$  is 
\[
g_\chi(z) = \sum_{\mathfrak a} \chi(\mathfrak a)  e^{2\pi i N(\mathfrak a) z} ,
\]
 summing  over all integral ideals $\mathfrak a$ of $K$ coprime to the conductor 
  of $\chi$,  where $z=x+iy$  with $x, y \in \R$ and  $y>0$ and $N(\mathfrak a)$ is the norm of $\mathfrak a$. 
For $\rho_\chi$ even, the Fourier expansion of the Maass cusp form $g_\chi$ also involves  the $K$-Bessel function $K_0$: 
\[
g_\chi(z) = \sum_{\mathfrak a} \chi(\mathfrak a) \sqrt{y}K_0(2\pi N(\mathfrak a) y) 2 \begin{cases} \cos(2\pi N(\mathfrak a) x),& {\rm if~} \rho_\chi(c) = Id; \\ \sin(2\pi N(\mathfrak a)x),& {\rm if~} \rho_\chi(c) = -Id .\end{cases} 
\]
 
Therefore, if $(\chi, K)$ is arithmetically equivalent to $(\eta, M)$ for quadratic extensions $K$ and $M$, then the modular form 
$g_\chi = g_\eta$ arises from  Hecke  characters of two different fields. An example was first found by Hecke, see \cite[page 243]{Serre}. In  section \S~\ref{sec:Dihedral forms}  we exhibit examples of this phenomenon.

From the viewpoint of Galois representations, Rohrlich studied such examples, called ``Hecke-Shintani representations'' in \cite{Rohrlich1}, in the course of deriving 
an asymptotic formula for the number of isomorphism classes of two-dimensional irreducible monomial representations of $G_\Q$ of bounded conductor.

\subsection{The method }
 
We interpret the equality $$L(s, \chi, K) = L(s, \eta, M)$$ on L-functions of characters as the equality $$L(s,\Ind_{G_K}^{G_\Q} \chi) = L(s, \Ind_{G_M}^{G_\Q} \eta)$$ on L-functions of induced degree-two representations $\Ind_{G_K}^{G_\Q} \chi$ and $\Ind_{G_M}^{G_\Q} \eta$ of the Galois group $G_{\Q}$. This then converts the problem on pair arithmetical equivalence to a problem on equivalence of induced representations. Our first main result, Theorem \ref{arithequiv} in \S2, gives a criterion for pair arithmetical equivalence in terms of the character involved. 

\begin{thm} \label{thm:sufficientcond1}
Let $K$ be a quadratic extension of $\Q$ with $\Gal(K/\Q) = \langle c \rangle$. Suppose a finite order character $\chi$ of $G_K$ is not equal to its conjugate $\chi^c$. Then the pair $(\chi, K)$ is arithmetically equivalent to another pair if and only if $\chi^c = \chi \cdot \delta$ for a quadratic character $\delta$ of $G_K$.
\end{thm}   
 We note that Rohrlich \cite{Rohrlich2} gave several criteria for non-arithmetically solitary pairs from the perspective of Galois representations.

In view of Theorem~\ref{thm:sufficientcond1}, our problem becomes one of finding 
finite order characters $\chi$ of $G_K$ satisfying
\begin{equation}\label{delta1}
\chi^c = \chi \cdot \delta
\end{equation}
for some quadratic character $\delta$ of $G_K$. From the viewpoint of idele class characters of $K$, a character $\mu$ is self-conjugate, that is, $\mu = \mu^c$, if and only if it comes from an idele class character of $\Q$ by composing with the norm map, in other words, it comes from base change. Therefore it suffices to consider finite order idele class characters $\chi$ of $K$ up to base change. As explained in \S3.3, we may further assume that the order of $\chi$ is a power of $2$.

We construct such $\chi$ for imaginary quadratic fields $\Q(\sqrt{-p})$ where $p$ is prime, $p=3\bmod 4$ or $p=2$, and also for $\Q(\sqrt{-1})$, and for the real quadratic fields $\Q(\sqrt{q})$ where $q=1\bmod 4$ is prime or $q=2$.   A key feature of these fields is that they have odd  class number. We also classify  
quadratic idele class characters $\chi$ up to base change for these quadratic extensions $K$ and determine their conductors. Note that for $\chi$ quadratic, $\chi \ne \chi^c$ if and only if $\chi$ and $\chi^c$ differ by a quadratic character; therefore we have classified, for such $K$, all pairs $(\chi, K)$ with quadratic $\chi$ arithmetically equivalent to another pair. The results for imaginary quadratic $K$ are given in Theorems \ref{2-classification} and \ref{2-Q(i)}.
The parallel result for real quadratic fields    is given in Theorem \ref{2-classification-Kreal}.

This paper is organized as follows. The main purpose of \S2 is to prove Theorem \ref{thm:sufficientcond1}. Proposition~\ref{Mackey} in \S2.1 establishes the necessity using Mackey's theory. \S2.2-\S2.5 explores the structure of the Galois group of a degree-two dihedral representation induced from a character $\chi$ satisfying (\ref{delta1}). Among other things, it provides information on the quadratic character $\delta$, and leads to the proof of sufficiency in \S2.6. An idele class character $\chi$ of $K$ of finite order gives rise to a primitive multiplicative character $\xi(\chi)$ of the quotient of the ring of integers $\mathbb Z_K$ of $K$ by the conductor $\mathfrak f_\chi$ of $\chi$. In \S3.1 we investigate the lifting problem: Given a nonzero integral ideal $\mathfrak a$ of $\mathbb Z_K$ and a primitive character $\xi$ of $(\mathbb Z_K/\mathfrak a)^\times$, when does $\xi= \xi(\chi)$ come from an idele class character $\chi$ of $K$? For quadratic fields $K$ with odd class number as specified above and $\xi$ of order a power of $2$, we obtain an easy condition, Corollary \ref{quadratic}, which is repeatedly used in the paper. \S3.2-\S3.3 concerns base change and reduction to characters of order powers of $2$, while \S3.4 describes possible conductors for such characters. With this information, in \S4 we characterize quadratic characters $\chi$ up to base change for quadratic fields $K$ specified above and determine their conductors, with imaginary fields in \S4.1 and real fields in \S4.2. Finally, in \S5 explicit examples of infinite families of cusp forms induced from quadratic characters of different quadratic fields are exhibited; holomorphic weight one forms are given in \S5.1, and Maass forms with different infinity type in \S5.2.    
 
\section{A criterion for pair arithmetical equivalences for quadratic extensions}

\subsection{A representation theoretic viewpoint of pair arithmetical equivalence}

Given a quadratic field $K/\Q$ and a finite order character $\chi$ of $G_K = \Gal(\bar \Q/K)$, we are interested in knowing whether there is another quadratic extension $M/\Q$ and a finite order   character $\eta$ of $G_M$ such that the two Artin $L$-functions agree:
$$ L(s, \chi, K) = L(s, \eta, M).$$
Induce the degree-one representation $\chi$ of $G_K$ to a degree-two representation $\Ind_K^{\Q}\chi = \Ind_{G_K}^{G_{\Q}} \chi$ of $G_{\Q}$. Since the Artin $L$-function is invariant under induction, we have
$$L(s, \chi, K) = L(s, \Ind_K^{\Q}\chi).$$
Similarly, for the pair $(\eta, M)$ we have the degree-two induced representation $\Ind_M^{\Q}\eta$, and
$$L(s, \eta, M) = L(s, \Ind_M^{\Q}\eta).$$
Hence if the two degree-$2$ representations are equivalent  
\begin{equation}\label{two equivalent}
 \Ind_K^{\Q}\chi \simeq \Ind_M^{\Q}\eta,
 \end{equation} 
then we have the desired equality
$$ L(s, \chi, K) = L(s, \eta, M).$$

This converts the question on arithmetical equivalence of pairs to a question on equivalence of representations. Suppose $$\Gal(K/\Q) = \langle c \rangle \qquad  {\rm and} \qquad \Gal(M/\Q) = \langle  \tau \rangle.$$ Then $c$ acts on $G_K$ by conjugation, which in turn defines the conjugate character $\chi^c$ by $\chi^c(h) = \chi(chc^{-1})$ for $h \in G_K$. Similarly define the conjugate $\eta^{\tau}$ of $\eta$.
In order that $\chi$ and $\eta$ induce irreducible representations of $G_{\Q}$, it is necessary and sufficient that the two characters are not self-conjugate, namely $\chi^c \ne \chi$ and $\eta^{\tau} \ne \eta$. 

\begin{prop}\label{Mackey} Suppose $\chi^c \ne \chi$ and $\eta^{\tau} \ne \eta$. Then  
\eqref{two equivalent} holds if and only if the restrictions to the subgroup $G_K\cap G_M= G_{KM}$ coincide:
\begin{equation}\label{restriction cond}
 \chi|_{G_K\cap G_M}  =  \eta|_{G_K\cap G_M}.   
\end{equation}
Moreover, we have
\[
\chi^{c} = \chi \cdot  \delta_{KM/K}, \quad \eta^\tau = \eta \cdot \delta_{KM/M}. 
\]
Here $\delta_{KM/K}$ and $\delta_{KM/M}$ are the quadratic characters of $\Gal(KM/K)$ and $\Gal(KM/M)$, respectively.
\end{prop}
\begin{proof} As Galois representations, $\chi$, $\eta$, $\Ind_K^{\Q} \chi$, and $\Ind_M^{\Q} \eta$ all factor through finite quotients of their respective Galois groups. So they may be viewed as representations of finite groups. 
Given a pair of   representations $\pi_1, \pi_2$ of a finite group $G$, denote by 
$$[\pi_1,\pi_2]_G = \dim \Hom_G(\pi_1,\pi_2).$$ 
So for $\pi_1,\pi_2$ irreducible, $[\pi_1,\pi_2]_G = 1$ if $\pi_1\simeq \pi_2$, and is equal to zero otherwise. 

We also recall Frobenius reciprocity for an induced representation from a subgroup $H$ of $G$:
\[
[ \Ind_H^{G} \xi , \pi]_G=[\xi ,\Res_H \pi]_H.
\]
Hence in our case
\[
[\Ind_{K}^{\Q}  \chi , \Ind_{M}^{\Q}  \eta]_{G_{\Q}}  = [\Res_{G_M} \Ind_{K}^{\Q}  \chi , \eta]_{G_M}.
 \]
Next we recall ``Mackey theory'', which says that the restriction of an induced representation has a direct sum decomposition
\begin{equation}\label{mackey formula}
\Res_{G_M} \Ind_{K}^{\Q}  \chi   \simeq \bigoplus_{s\in G_K\backslash G_{\Q}/G_M} \Ind_{s^{-1}G_Ks\cap G_M}^{G_M} (\chi^{s} ),
\end{equation}
where $\chi^{s}(h) = \chi(shs^{-1})$. 

In our case, $G_K,  G_M$ have index two in $G_{\Q}$, hence are normal, and moreover $G_KG_M=G_{\Q}$ since $G_K\neq G_M$. Hence there is only one double coset, and Mackey's formula \eqref{mackey formula} reduces to 
\[
\Res_{G_M} \Ind_{K}^{\Q}  \chi   \simeq \Ind_{G_K\cap G_M}^{G_M}\chi.
\]
Hence
\[
[\Ind_{K}^{\Q}  \chi , \Ind_{M}^{\Q}  \eta]_{G_{\Q}}  = [\Res_{G_M} \Ind_{K}^{\Q}  \chi , \eta]_{G_M}
 =[  \Ind_{G_K\cap G_M}^{G_M}\chi,\eta]_{G_M}.
\]
Applying again Frobenius reciprocity gives
\[
[  \Ind_{G_K\cap G_M}^{G_M}\chi,\eta]_{G_M} = [\chi,\eta]_{G_K\cap G_M} = \begin{cases} 1,&{\rm if}~ \chi|_{G_K\cap G_M} = \eta|_{G_K\cap G_M} \\0,&\mbox{otherwise, } \end{cases}
\]
which proves the claim \eqref{restriction cond}.

Moreover, we have $\Ind_{K}^{\Q}\chi = \Ind_{K}^{\Q} \chi^c$, so the same conclusion holds with $\chi$ replaced by $\chi^c$, and in particular we must have
\[
\chi|_{G_K\cap G_M} = \chi^c|_{G_K\cap G_M}.
\]
Since $[G_K:G_K\cap G_M]=2$, this means that $\chi^{-1}\chi^c$ is a quadratic character of $G_K$, nontrivial because $\chi\neq \chi^c$ for irreducibility, which is trivial on $G_K\cap G_M = G_{KM}$.  Hence it must equal $\delta_{KM/K}$.  

Likewise $\eta^{-1}\eta^\tau$ is a quadratic character of $G_M$, which is trivial on $G_K\cap G_M$. Hence it must equal $\delta_{KM/M}$.  
\end{proof}

Proposition \ref{Mackey} provides a necessary condition for $\Ind_K^{\Q} \chi$ to be equivalent to another induced representation of the same type, namely $\chi^c$ differs from $\chi$ by a quadratic character. The theorem below says that this condition is also sufficient.

\begin{thm}\label{arithequiv} Let $K/\Q$ be a quadratic extension with $\Gal(K/\Q) = \langle c \rangle$. Let $\chi$ be a finite order non self-conjugate character of $G_K$. Then 

(a) $\Ind_K^{\Q} \chi \simeq \Ind_M^{\Q} \eta$ for some pair $(\eta, M)$ with $M \ne K$

\noindent if and only if 

(b) $\chi^c = \chi \cdot \delta$ for some quadratic character $\delta$ of $G_K$.
\end{thm}

As discussed above, (a) implies $(\chi, K)$ is arithmetically equivalent to $(\eta, M)$. Hence
we have the following immediate corollary, which provides a convenient sufficient condition for $(\chi, K)$ to be arithmetically equivalent to another pair.

\begin{cor}\label{sufficientcond} Let $K/\Q$ be a quadratic extension with $\Gal(K/\Q) = \langle c \rangle$. Let $\chi$ be a finite order character of $G_K$. If $\chi^c = \chi \cdot \delta$ for some quadratic character $\delta$ of $G_K$, then there is a pair arithmetically equivalent to $(\chi, K)$. 
\end{cor}

Note that the non self-conjugate requirement for $\chi$ automatically follows from the condition $\chi^c = \chi \cdot \delta$. 

Theorem \ref{arithequiv} will be proved in the last subsection of this section, after studying the structure of the Galois group of a dihedral representation induced by a character $\chi$ satisfying the condition (b).

\subsection{The Galois group of a dihedral representation}

Consider a two-dimensional \underline{faithful}\footnote{otherwise it factors through a subfield of $E$, that is through $\Gal(E'/\Q)$ for $\Q\subset E'\subset E$} 
irreducible representation $$\rho:G:=\Gal(E/\Q)\to \GL(2,\C)$$ where $E/\Q$ is a finite extension, which is {\em dihedral}, in the sense that the projectivation 
 \[
 \bar \rho: G\to \GL(2,\C) \to \PGL(2,\C)
 \]
 is a dihedral group $D_n$ of order $2n$: 
 \[
 \bar\rho(G) \simeq D_n.
 \]
Thus  there is an index-two subgroup $H\subset G$ and a character $\chi:H\to \C^\times$ so that 
  $$\rho\simeq \Ind_H^G(\chi).$$ 
    
    The subgroup $H$ corresponds to a quadratic extension $K/\Q$, a subfield of $E$, namely the fixed points of $H$ in $E$, so that 
    \[
    H=\Gal(E/K).
    \] 
    Let $c$ be an element of $\Gal(E/\Q)$ such that when restricted to $K$ it gives the Galois involution of $K$ so that we can write 
  \[
  \Gal(K/\Q) = \{1,c\} \simeq \Gal(E/\Q)/\Gal(E/K) = G/H. 
  \] 

 Assume further that there is a quadratic character $\delta:H\to \{\pm 1\}$, so that
 \[
 \chi^c = \chi \cdot \delta, 
 \]
 where $\chi^c$ is the character on $H$ given by $\chi^c(h) = \chi(chc^{-1})=\chi(c^{-1}hc)$ (since $c^2 \in H$) arising from the short exact sequence 
 \[
 1 \rightarrow H \rightarrow G \rightarrow \Gal(K/\Q) \rightarrow 1.
 \]
We study the structure of $G=\Gal(E/\Q)$.
 \begin{thm}\label{centerofG}
  The Galois group   $G=\Gal(E/\Q)$ is of order $4m $, with the center $Z =$ ker $\delta$ being a cyclic subgroup of   order $m $. 
The projectivization $G/Z\simeq \bar\rho (G)\subset \PGL(2,\C)$ is the Klein 4-group $D_2\simeq \Z/2\Z \oplus \Z/2\Z$. 
 \end{thm}

 \begin{proof} Let ${\bf v_1}$ be a basis of the $1$-dimensional representation $\chi$ of $H$. Then ${\bf v_2} = \rho(c)({\bf v_1})$ is linearly independent of ${\bf v_1}$ since $G$ is generated by $H$ and $c$, and $\rho$ is $2$-dimensional.
In the basis ${\bf v_1, \bf v_2}$ of $\Ind_H^G\chi$, we have 
\[
\rho(c) = \begin{pmatrix} 0& \chi(c^2) \\ 1 &0 \end{pmatrix}\quad{\rm and}
 \quad 
\rho(h) = \begin{pmatrix} \chi(h)& 0\\ 0&\chi^c(h) \end{pmatrix} ~{\rm for}~h \in H
\]
because $c^2 \in H$.  

Since we assume that $\chi^c = \chi \cdot \delta$, we have
\[
\rho(h) = \chi(h)  \begin{pmatrix} 1& 0\\ 0&\delta(h) \end{pmatrix}
\]
and since $\delta$ is a quadratic character, we have that on the index two subgroup $\ker \delta\subset H$, $\rho(h)$ is a scalar matrix
\[
\rho(h) = \chi(h)I, \quad h\in \ker \delta.
\]

Therefore the center of $G$ equals $\ker \delta$, which has index $4$ in $G$: 
\[
{\rm Center}(G) = \ker \delta
\]
and the image in $\PGL(2,\C)$ of $\rho$ is therefore a group of order $4$, isomorphic to $G/\ker\delta$. Since we are given that it is dihedral, we therefore conclude that it equals $D_2$.

We note that since $\rho = \Ind_H^G\chi$ is faithful, the same holds for its restriction to $H$, that is, the subgroup
 \[
 \rho(H) = \{ \chi(h) \begin{pmatrix} 1&0 \\0&\delta(h) \end{pmatrix} : h\in H\}.
 \]

We now explore the implications of our conditions on the structure of $G$, knowing that the center is $Z= \ker \delta$, a subgroup of index $2$ in $H$. In particular the order of $H$ is even: $|H|=2m$. 

The center is $Z=\ker\delta$, the image under $\rho $ being 
\[
\rho(Z) = \{\chi(h)I: h\in Z\}
\]
and the restriction of $\chi$ to the center $Z$ must be faithful, so $\chi(Z)$ being a finite subgroup of the multiplicative group of the field of complex numbers must be cyclic, say is cyclic of order $m$: $Z=\langle z_0: z_0^m=1\rangle$ and $G/Z\simeq \Z/2\Z\oplus \Z/2\Z$. 
 \end{proof}

It remains to restrict the structure of $G$.   

\begin{thm}\label{structureofG}
i) $m$ is even: $m=2\mu$, so that $|G|=4m=8\mu$ has order divisible by $8$. 

ii) The action of 
$c$ on $H$ is given by 
\[
chc^{-1} = \begin{cases} h, & h\in Z\\ z_1h,&h\notin Z \end{cases}
\]
where $z_1\in Z$ is the unique involution in $Z$ (corresponding to the element $m/2\in \Z/m\Z$). 

iii) There are two possibilities for H

a)  $H\simeq \Z/2m\Z$ is cyclic; 

b) $H=\Z/m\Z\oplus \Z/2\Z$. 

\end{thm}

To prove this theorem, there are two possibilities to consider: $\chi$ is a faithful character of $H$, or not, which are carried out in the next two subsections.

\subsection{The case $\chi$ is faithful}
In this case $H\simeq \chi(H)$ is embedded as a finite subgroup of the multiplicative group of the complex numbers, hence is cyclic, of order 
$2m$, say $H=\langle h_0: h_0^{2m}=1 \rangle$, then $Z = \langle h_0^2\rangle$ consists of the squares in $H$. 

Now conjugation by $c$ acts by an automorphism, hence takes the generator $h_0$ to $h_0^k$: 
\[
ch_0c^{-1} = h_0^{k}
\]
with $k$ coprime to $|H|=2m$. In particular $k$  is odd. 

We can obtain further information by using the matrix representation: Recall $\delta(h_0) = -1$, and $\chi(h_0)=\zeta_{2m}$ is a primitive $2m$-th root of unity, and 
\[
\rho(h_0)= \chi(h_0) \begin{pmatrix} 1& 0\\ 0&-1 \end{pmatrix}  = \zeta_{2m}\begin{pmatrix} 1& 0\\ 0&-1 \end{pmatrix}. 
\]
We have
\[
\rho(c)\rho(h_0) \rho(c)^{-1} = \begin{pmatrix} 0&\chi(c^2)\\ 1&0 \end{pmatrix}\zeta_{2m} \begin{pmatrix} 1& 0\\ 0&-1 \end{pmatrix}\begin{pmatrix} 0&1\\ \chi(c^2)^{-1}&0 \end{pmatrix} =\zeta_{2m} \begin{pmatrix} -1& 0\\ 0&1 \end{pmatrix}
\]
 while on the other hand
 \[
 \rho(c)\rho(h_0) \rho(c)^{-1} =  \rho(ch_0c^{-1}) = \rho(h_0^k) =\rho(h_0)^k =  \zeta_{2m}^k  \begin{pmatrix} 1& 0\\ 0& (-1)^k \end{pmatrix}
 \] 
 giving 
 \[
\zeta_{2m}^{k-1} = -1.
 \]
Since  $\zeta_{2m}$ is a primitive $2m$-th root of unity,   we find  
 \[
 k-1 = m\mod 2m  
 \]
 so that the only possibility is (given $1\leq k<2m$)
 \[
 k= m+1.
 \]
 Combined with the fact that $k$ is odd, as observed above, this implies that $m=2\mu$ is even. 

 We can rewrite the action of $c$ on $H$ as follows: As $c$ commutes with the center $Z$, which has index $2$ in $H$, it suffices to compute   
 \[
 ch_0c^{-1} = h_0^{m+1} = h_0 \cdot (h_0^2)^{m/2}=h_0 z_1, 
 \]
 where $z_1 = z_0^{m/2}$ is the unique element of order $2$ in $Z$. 
 This proves 
 \[
 chc^{-1} = \begin{cases} h,& h\in Z\\ hz_1,&h\notin Z, \end{cases}
 \] as described in (ii).

\subsection{The case $\chi$ is not  faithful}  
Recall that $Z=\langle z_0: z_0^m=1\rangle$ is cyclic of order $m$, so that $\rho(z_0) = \zeta_mI$ for $\zeta_m$ a primitive $m$th root of unity. Since $\chi(H) = \chi(Z)$ by assumption and $\chi$ is faithful on $Z$, ker~$\chi$ is generated by an element $h_1 \in H$ of order $2$. As $Z$ commutes with $h_1$ and $Z$ intersects $\langle h_1 \rangle$ trivially, we have 
\[
H=Z\coprod h_1 Z = Z \times \langle  h_1 \rangle = {\rm ker}~ \delta \times {\rm ker}~ \chi \cong \Z/m\Z \times \Z/2\Z.
\]
The action of $\rho$ on $H$ is given by  
\begin{equation}\label{Matrices for H}
\rho(z_0^j) =  \zeta_m^jI, \qquad \rho(h_1z_0^j) = \zeta_m^j \begin{pmatrix} 1&0\\0&-1\end{pmatrix},\qquad 0\leq j\leq m-1. 
\end{equation}

We prove that $m$ is even: Conjugating by $c$, we must have $ch_1c^{-1}\in H$, so that 
\[
\rho(c)\rho(h_1)\rho(c)^{-1}  = \begin{pmatrix} 0& \chi(c^2) \\ 1 &0 \end{pmatrix}\begin{pmatrix} 1&0\\0&-1\end{pmatrix}
  \begin{pmatrix} 0& 1 \\ \chi(c^2)^{-1} &0 \end{pmatrix} = - \begin{pmatrix} 1&0\\0&-1\end{pmatrix} \in \rho(H),
\]
that is,
\[
-\begin{pmatrix} 1&0\\0&-1\end{pmatrix} = \zeta_m^j \begin{pmatrix} 1&0\\0&-1\end{pmatrix}
\]
for some $j$, forcing $\zeta_m^j=-1$. But if $m$ is odd then $-1$ is not an  $m$-th root of unity, giving a contradiction. Hence $m$ must be even. Furthermore, the above computation shows that
%

\[
\rho(c)\rho(h_1)\rho(c)^{-1}  = -\rho(h_1) = \zeta_m^{m/2} \rho(h_1) = \rho(z_0^{m/2}h_1).
\]
Together with $c$ commuting with $Z$, we obtain the action of $c$ on $H$:
 \[
 chc^{-1} = \begin{cases} h,& h\in Z\\ hz_1,&h\notin Z, \end{cases}
 \]
where $z_1 = z_0^{m/2}$ is the unique element of order $2$ in $Z$, as described in (ii). 
\bigskip

\subsection{Completing the proof of Theorem \ref{structureofG}}
In conclusion,  $G$ is a semi-direct product  
\[
G= H\rtimes  \Z/2\Z 
\] with $H = \Z/2m\Z$ if $\chi$ is faithful on $H$ and $H = \Z/m\Z \times \Z/2\Z$ otherwise. This proves (iii). The evenness of $m$ is also established in both cases. This is (i). The proof of Theorem \ref{structureofG} is now completed.
\medskip

Several remarks are in order.
 
{\bf Remarks.} 1. Since the action of $c$ fixes the elements in $Z$ and interchanges $h$ and $z_1h$ for $h$ outside $Z$, this shows that $\Frob_v$ for the primes $v$ above $p$ inert in $K$ are in $Z$ and those above $p$ splitting in $K$, if not identity, are outside $Z$. 

2. For $h = c^2 \in H$, since $chc^{-1} = c^2 = h$, we have $c^2 \in Z$ by Theorem \ref{structureofG} (ii). Therefore $\delta(c^2) = 1$ because ker $\delta = Z$. 

3. The fixed field $F$ of ker $\delta = Z$ is a biquadratic extension of $\Q$ since $G/Z$ is a Klein-$4$ group by Theorem \ref{centerofG}. Thus $F$ contains three quadratic extensions of $\Q$, one of which is $K$, and $F=KM$ for any quadratic subfield $M \ne K$. Moreover, $\delta$, being the unique quadratic character on $\Gal(KM/K)$, is equal to the quadratic character $\delta_{KM/K}$ attached to the extension $KM/K$, which lifts the quadratic character $\delta_{M/\Q}$. This shows that  
$$ \chi^c = \chi \cdot  \delta_{KM/K}$$
for any quadratic subfield $M$ of $F$ other than $K$. In particular, since $c$ restricted to $F$ has order $2$, we may choose $M$ to be the fixed subfield of $c$ on $F$.

4. Suppose $\chi$ has order $r$. Raising both sides of $\chi^c = \chi \cdot \delta$ to the $r$th power implies $r$ even since $\delta$ has order $2$. This is in concert with statement (i) in Theorem \ref{structureofG}.

The following proposition gives different criteria for the faithfulness of $\chi$ on $H$.

\begin{prop}\label{faithful} Suppose $\chi$ has order $r \equiv 0 \mod 4$. The following statements are equivalent:

(a) $\chi$ is faithful on $H$.

(b) ker $\delta$ contains ker $\chi$, equivalently, the fixed field of ker $\chi$ contains that of ker $\delta$. 

(c) $\chi^{r/2}$ is $\delta$.

(d) $\chi^c = \chi^{1 + r/2}$.
\end{prop}

Note that the assertion (d) is a condition for faithfulness of $\chi$ on $H$ in terms of $\chi$ alone. 

\begin{proof} Since $\chi^c = \chi \cdot \delta$, we have (c) $\Longleftrightarrow$ (d). 

(a) $\Rightarrow$ (c). Since $\chi$ is faithful on $H$, so $H \cong \chi(H)$ is cyclic and $\chi^{r/2}$, being the unique character on $H$ of order $2$, is equal to $\delta$. 

(c) $\Rightarrow$ (b). This is because ker $\chi^{r/2}$ contains ker $\chi$.

(b) $\Rightarrow$ (a). It is shown in \S2.4 that, if $\chi$ is not faithful on $H$, then ker $\chi = \langle  h_1 \rangle$ has order $2$ and ker $\delta = Z$ intersects ker $\chi$ trivially. So ker $\delta$ does not contain ker $\chi$. 
\end{proof}

As an immediate consequence, we have the following description of the fixed field $E$ of $H$.

\begin{cor} $E$ is the composition of the fixed field of ker $\chi$ and that of ker $\delta$. 
\end{cor}

\subsection{A proof of Theorem \ref{arithequiv}}

 By Proposition \ref{Mackey}, (a) implies (b) with $\delta = \delta_{KM/K}$, a quadratic character of $\Gal(KM/K)$ and hence of $G_K$. 

Now we consider the converse. Suppose $\chi^c = \chi \cdot \delta$ for some quadratic character $\delta$ of $G_K$. We shall find a quadratic extension $M$ of $\Q$ different from $K$ and a finite order non self-conjugate character $\eta$ of $G_M$ such that $\chi|_{G_K\cap G_M} = \eta|_{G_K \cap G_M}$, which in turn implies (a) by Proposition \ref{Mackey}. 

Write $\rho$ for the induced representation $\Ind_K^{\Q} \chi$, which is an irreducible degree-$2$ dihedral representation of $G_{\Q}$ since $\chi$ is not self-conjugate. Then by Remark 3 in \S2.5, there is a quadratic subfield $M \ne K$ of the fixed field of ker $\delta$, pointwisely fixed by $c$, such that we may choose $\delta = \delta_{KM/K}$. We proceed to find a character $\eta$ of $G_M$ with the desired properties.  

Let ${\bf v_1}$ be a basis of the space of $\chi$ and ${\bf v_2} = \rho(c)({\bf v_1})$ so that $\{\bf v_1, \bf v_2\}$ is a basis of the space $V$ of $\rho$. In this basis, we have 
\[
\rho(c) = \begin{pmatrix} 0& \chi(c^2) \\ 1 &0 \end{pmatrix}\quad{\rm and}
 \quad 
\rho(h) = \begin{pmatrix} \chi(h)& 0\\ 0&\chi^c(h) \end{pmatrix} ~{\rm for}~h \in G_K
\]
as in the proof of Theorem \ref{centerofG}. Since ker $\delta_{KM/K} = G_{KM}$, we have $\chi = \chi^c$ on $G_{KM}$ and $\chi = -\chi^c$ on $G_K \setminus G_{KM}$ by assumption. 

As $c$ fixes $M$ elementwise, $c \in G_M$ and $G_M = G_{KM} \cup G_{KM}c$. We study the restriction of $\rho$ to $G_M$. By Remark 2 of the previous subsection, $c^2$ lies in $G_{KM}$ so that $\rho(c^2)$ is scalar multiplication by a nonzero constant $a = \chi(c^2)$. Write $a = b^2$. Let ${\bf w_1} = {\bf v_1} + b^{-1}{\bf v_2}$ and ${\bf w_2} = {\bf v_1} - b^{-1}{\bf v_2}$. In terms of the new basis $\{{\bf w_1}, {\bf w_2}\}$ of $V$, we find $\rho(c){\bf w_1} = b{\bf w_1}$ and $\rho(c){\bf w_2} = -b{\bf w_2}$. Let $h \in G_{KM}$. With respect to the basis $\{{\bf w_1}, {\bf w_2}\}$ of $V$, the action of $\rho(h)$ is given by scalar multiplication $\left(\begin{matrix}\chi(h) & 0\\0 & \chi(h)\end{matrix}\right)$ as discussed before, while the action of $\rho(hc)$ is represented by $\left(\begin{matrix}b\chi(h) & 0\\ 0 & -b\chi(h)\end{matrix}\right)$, analogous to  $\rho$ restricted to $G_K$.  
This shows that the $1$-dimensional space spanned by ${\bf w_1}$ is invariant under $\rho(G_M)$ and hence the action is given by a character $\eta$ of $G_M$, and 
$\rho$ restricted to $G_M$ is a direct sum of two  characters $\eta$ and $\eta \cdot \delta_{KM/M}$ with $\eta|_{G_K \cap G_M} = \chi|_{G_K \cap G_M}$. Moreover, $\eta$ has finite order and $\eta$ is not self-conjugate, as desired.

\section{Characters of Galois groups and idele class characters} 
\subsection{Idele class characters}
We first set up notation and recall some facts. For a number field $K$, denote by $\Z_K$ its ring of integers and $U_K$ the group of units in $\Z_K$. The maximal ideals of $\Z_K$ give rise to finite places of $K$, while the infinite  places of $K$ come from the $r_1$ distinct real imbeddings and $r_2$ nonconjugate complex imbeddings of $K$. Together, they constitute the set $\Sigma(K)$ of all places of $K$. The completion of $K$ at a place $v \in \Sigma(K)$ is denoted $K_v$. When $v$ is finite, let $\mathcal O_v$ denote the ring of integers in $K_v$, $\mathcal U_v$ the group of units, and $\mathcal M_v$ the unique maximal ideal of $\mathcal O_v$. Any generator $\pi_v$ of $\mathcal M_v$ is a uniformizer of $K_v$. Then $K_v^\times = \mathcal U_v \times \langle  \pi_v \rangle$. Clearly $\mathcal U_v$ contains the group of global units $U_K$.   

The (topological) group of ideles of $K$, defined by 
$$ I_K = \{x=(x_v) \in \prod_{v \in \Sigma(K)} K_v^\times~|~ x_v \in \mathcal U_v~{\rm for~almost~all}~v \},$$ is the restricted product of $\{K_v^\times : v \in \Sigma(K)\}$ with respect to $\{\mathcal U_v : {\rm finite}~v \in \Sigma(K)\}$. 
The field $K^\times$ is diagonally imbedded in $I_K$ as a discrete subgroup, and the quotient $I_K/K^\times$ is called the idele class group of $K$. We may write $I_K = (I_K)_{\infty}(I_K)^{\infty}$, where $$(I_K)_{\infty}= \prod_{v \in \Sigma(K)~{\rm infinite}}K_v^\times$$ is the subgroup of ideles supported at the infinite places and $(I_K)^\infty$ is the subgroup of ideles supported at the finite places. 
 
A character $\chi$ of $I_K$ is a continuous homomorphism from $I_K$ to the unit circle in $\mathbb C^\times$. It can be expressed as a product of $\chi_v$, its restriction to $K_v^\times$, over all $v \in \Sigma(K)$. Note that $\chi_v$ at a finite place $v$ is determined by its values on the group of units $\mathcal U_v$ and one uniformizer $\pi_v$. Sometimes we write $\chi = \chi_\infty \chi^\infty$, where $\chi_\infty$ and $\chi^{\infty}$ are the restrictions of $\chi$ to $(I_K)_{\infty}$ and $(I_K)^{\infty}$, respectively. We discuss $\chi^\infty$ and $\chi_\infty$ in more detail below.

As a result of continuity, for almost all finite $v$, $\chi_v(\mathcal U_v) = 1$, in which case we say that $\chi_v$ is unramified, or $\chi$ is unramified at $v$. The set $S$ of finite places where $\chi$ is ramified is finite. At each $v \in S$, there is a smallest positive integer $n(\chi_v)$ such that $\chi_v$ is trivial on $1 + \mathcal M_v^{n(\chi_v)}$; call $\mathcal M_v^{n(\chi_v)}$ the conductor of $\chi_v$. The conductor of $\chi$ is the product $\ff_\chi = \prod_{v \in S} v^{n(\chi_v)}$, which is a nonzero ideal of $\mathbb Z_K$. Observe that $\prod_{v \in S} \mathcal U_v/(1 + \mathcal M_v^{n(\chi_v)}) \cong (\mathbb Z_K/\ff_\chi)^\times$, hence $\chi$ restricted to $\prod_{v \in S} \mathcal U_v$ induces a character $\xi=\xi(\chi)$ of $(\mathbb Z_K/\ff_\chi)^\times$. Moreover, since $\ff_\chi$ is the conductor of $\chi$, $\xi$ is a primitive character of $(\mathbb Z_K/\ff_\chi)^\times$ in the sense that it does not induce a character of $(\mathbb Z_K/\mathfrak a)^\times$ for any ideal $\mathfrak a$ of $\mathbb Z_K$ properly containing $\ff_\chi$. In conclusion, $\chi^\infty$ on the group of units $\prod_{v ~{\rm finite}} \mathcal U_v$ lifts a primitive character $\xi$ on $(\mathbb Z_K/\ff_\chi)^\times$, where $\ff_\chi$ is the conductor of $\chi$, and $\chi^\infty$ is determined by $\xi$ and the values $\chi_v(\pi_v)$ for all finite places $v$.  

Now {\it assume $\chi$ has finite order}. We discuss $\chi_\infty$. If $v$ is a real place of $K$, then $K_v = \mathbb R$ and $\chi_v$ is either trivial or the sign function on $\mathbb R^\times$; while if $v$ is a complex place, then $K_v = \mathbb C$ and $\chi_v$ is always trivial on $\mathbb C^\times$.

A character $\chi$ of $I_K$ is called an idele class character of $K$ if it is trivial on $K^\times$. This is a strong constraint on $\chi$. For example, $\chi_\infty$ and $\xi(\chi)$ are related by $\chi(U_K)=1$. Therefore, an idele class character $\chi$ of $K$ is determined by $\chi_\infty$, $\xi(\chi)$ and $\chi_v(\pi_v)$ for finitely many $v$ representing the ideal class group of $K$. In particular, if $K$ has class number one, then $\chi$ is determined by $\chi_\infty$ and $\xi(\chi)$. 

The theorem below gives other occasions that an idele class character $\chi$ is determined by $\xi(\chi)$ and $\chi_\infty$.

\begin{thm}\label{extension} Let $\fa$ be a nonzero ideal of the ring of integers $\Z_K$ of a number field $K$. Let $\xi$ be a primitive character of $(\Z_K/\fa)^\times$ of even order $r$ which lifts to a character $\chi_U$ on the group of units in $(I_K)^\infty$. Suppose that there is a character $\chi_\infty$ on $(I_K)_\infty$ with order $\le 2$ such that the character $\chi_J :=\chi_\infty \times \chi_U$ on 
$J := (I_K)_\infty \prod_{v \in \Sigma(K)~{finite}} \mathcal U_v$ is trivial on the group of global units $U_K$. If the order $r$ of $\xi$ is coprime to the class number of $K$, then $\chi_J$ on $J$ has a unique extension to a character $\chi$ on $I_K$ of order $r$, conductor $\ff_\chi = \fa$,  and trivial on $K^\times$.
\end{thm}  
\begin{proof} At each finite place $v$ of $K$, fix a uniformizer $\pi_v$. To extend $\chi_J$ to a character $\chi = \prod_{v \in \Sigma(K)} \chi_v$ on $I_K$, it remains to define $\chi^\infty(\pi_v) = \chi_v(\pi_v)$ so that $\chi$ is trivial on $K^\times$. It will be clear from the definition that the resulting $\chi$ has order $r$ and conductor $\ff_\chi = \fa$. By assumption, the class number $h = h(K)$ is coprime to the order $r$ of $\chi$, so there is a positive integer $e$ such that $eh \equiv 1 \mod r$. By construction, $\chi_J$ has order $r$, hence $(\chi_J)^{he} = \chi_J$. We shall take advantage of the fact that each ideal of $\mathbb Z_K$ raised to the $h$-th power is principal (since $h$ is the class number of $K$) to facilitate our definition of $\chi_v(\pi_v)$. 

For each finite place $v$ of $K$, choose an element $\beta_v \in \mathbb Z_K$ which generates the ideal $v^h$, i.e., $v^h = (\beta_v)$. Note that $\beta_v$ is a unit at all finite places $w \in \Sigma(K)$ outside $v$, and at $v$, $\beta_v = u_v \pi_v^h$ for some unit $u_v$ in $K_v$. For $v \nmid \fa$, define 
$$\chi_v(\pi_v) =\chi_\infty(\beta_v)^{-e}\prod_{w|\fa}\chi_w(\beta_v)^{-e}$$
so that $\chi_v(\pi_v)^h = \chi_\infty(\beta_v)^{-1}\prod_{w|\fa}\chi_w(\beta_v)^{-1}$ and 
$$\chi(\beta_v) = \chi_\infty(\beta_v)\cdot \prod_{w|\fa}\chi_w(\beta_v)\cdot \chi_v(\beta_v)= \chi_\infty(\beta_v)\cdot\prod_{w|\fa}\chi_w(\beta_v)\cdot \chi_v(\pi_v)^h =1;$$  
while for $v | \fa$, define
$$\chi_v(\pi_v) = \chi_\infty(\beta_v)^{-e}\chi_v(u_v)^{-e}\prod_{w | \fa, ~w \ne v}\chi_w(\beta_v)^{-e}$$
so that $\chi_v(\pi_v)^h = \chi_\infty(\beta_v)^{-1}\chi_v(u_v)^{-1}\prod_{w | \fa, ~w \ne v}\chi_w(\beta_v)^{-1}$ and 
$$\chi(\beta_v) = \chi_\infty(\beta_v)\prod_{w | \fa}\chi_w(\beta_v) = \chi_\infty(\beta_v)\chi_v(u_v \pi_v^h)\prod_{w|\fa,~w\ne v}\chi_w(\beta_v) = 1.$$
Thus we have extended $\chi_J$ to a character $\chi$ of $I_K$ trivial on $U_K$ of order $r$, and with conductor $\ff_\chi = \fa$. Moreover $\chi^h$ is trivial on nonzero elements in $\Z_K$ and hence on $K^\times$. Raising it to the $e$-th power shows that $\chi$ is trivial on $K^\times$, as desired.

To prove the uniqueness of $\chi$, let $\chi'$ be another extension of $\chi_J$ with the desired properties. Then $\mu = \chi' \chi^{-1}$ is an idele class character of $K$ which is trivial on $J$. Hence $\mu_\infty$ is trivial, $\mu$ is unramified everywhere and has order dividing $r$. Therefore, at a finite place $v$, we have $1 = \mu(\beta_v) =\mu_v(\beta_v) = \mu_v(\pi_v)^h$ since $\beta_v$ is a unit outside $v$. Raising it to the $e$-th power gives $\mu_v(\pi_v) = 1$ for all finite $v$. Thus $\mu$ is trivial, in other words, $\chi = \chi'$. 
\end{proof}

We illustrate some constraints on $\xi(\chi)$ for an idele class character.

\begin{prop}\label{constraint-on-xi} Let $\chi$ be an idele class character of $K$ of finite order with conductor $\ff_\chi$. Let $\xi=\xi(\chi)$ be the primitive character of $(\Z_K/\ff_\chi)^\times$ so that $\chi^\infty$ on units in $(I_K)^\infty$ is the lift of $\xi$. Then the following hold. 

(i) $\xi(u) = \chi_\infty(u)$ for each $u \in U_K$, $\xi(U_K)\subseteq \langle -1 \rangle$, and $\xi$ is trivial on units in $U_K$ which are positive under all real imbeddings of $K$.

(ii) If $K$ is a CM field, then $\chi_\infty$ is trivial and $\xi(U_K)=1$.

(iii) For $K$ real quadratic, if there is a fundamental unit $\epsilon_K$ with norm $N_{K/\Q}(\epsilon_K) = -1$, then $\chi_\infty$ is uniquely determined by $\xi(\epsilon_K)$ and $\xi(-1)$; if $N_{K/\Q}(U_K)=1$, then $\chi_\infty$ is determined by $\xi(-1)$ and $\chi^\infty(\alpha)$ for any $\alpha \in K^\times$ with negative norm. 
\end{prop}
\begin{proof} Note that $\chi^\infty(u) = \xi(u)$ for any $u \in U_K$ since $\chi$ is unramified outside the support of $\ff_\chi$. 

(i) For a general $K$, $\chi_v$ at each infinite place $v$ has order at most $2$ so that $\chi_\infty^2$ is trivial. Thus $\xi(u) = \chi_\infty(u)^{-1} = \chi_\infty(u)$ for each $u \in U_K$.
Hence $\chi_\infty(-1) = \xi(-1).$ Further, since $\chi_v$ at a real place $v$ is either trivial or the sign function, $\chi_v(u) = 1$ if $u$ is positive under the imbedding $v$. This shows that $\xi(u) = \chi_\infty(u) = 1$ if the unit $u$ has positive images under all real imbeddings of $K$. 

(ii) Suppose $K$ is a CM field. Then all infinite places of $K$ are complex and hence $\chi_\infty$ is trivial.  Given $u \in U_K$, from $\chi(u) = \chi_\infty(u)\chi^\infty(u) = 1$ we conclude  $\chi^\infty(u) = \xi(u)=1$. So $\xi(U_K) = 1$.

(iii) Now assume $K$ is real quadratic. The group of global units $U_K = \langle -1 \rangle \times\langle  \epsilon_K \rangle$, where $\epsilon_K$ is a fundamental unit of infinite order. The field $K$ has two real imbeddings $\infty_1$ and $\infty_2$, and each $\chi_{\infty_i}$ is either trivial or the sign function on $\mathbb R^\times$. From $\chi_\infty(-1)=\xi(-1)=\pm 1$ we conclude 
$\chi_{\infty_1} \ne \chi_{\infty_2}$ if $\xi(-1) = -1$, and $\chi_{\infty_1} = \chi_{\infty_2}$ if $\xi(-1) = 1$. If $N_{K/\Q}(\epsilon_K) = \infty_1(\epsilon_K)\infty_2(\epsilon_K) = -1$, then exactly one of $\infty_1(\epsilon_K), \infty_2(\epsilon_K)$, say, $\infty_1(\epsilon_K)$, is positive so that $\chi_{\infty_1}(\epsilon_K)=1$ always holds, and $\chi_{\infty_2}$ is determined by $\chi_{\infty_2}(\epsilon_K)= \chi_{\infty_1}(\epsilon_K)\chi_{\infty_2}(\epsilon_K) = \chi_\infty(\epsilon_K)= \xi(\epsilon_K)$. This in turn determines $\chi_{\infty_1}$ so that $\chi_\infty (-1) = \xi(-1)$. Hence $\chi_\infty$ is uniquely determined by $\xi(-1)$ and $\xi(\epsilon_K)$ in this case. On the other hand, if $N_{K/\Q}(U_K) = 1$, the information of $\xi$ on $U_K$ is not enough to determine $\chi_\infty$. Instead, we choose any $\alpha \in K^\times$ with negative norm. Then $\chi^\infty(\alpha) = \chi_\infty(\alpha) = \pm 1$ since $\chi(\alpha) = 1$. The above argument for $\epsilon_K$ with norm $-1$ is easily modified to pin down $\chi_\infty$.  
\end{proof} 

Suppose $K$ is quadratic over $\Q$. In particular, given a primitive character $\xi$ on $(\Z_K/\fa)^\times$ satisfying the requirement (ii) for $K$ imaginary quadratic, and requirement (i) for $K$ real quadratic containing a fundamental unit with negative norm, 
 then, by the proposition above, it determines a unique $\chi_\infty$ on $(I_K)_\infty$ such that the character $\chi_J$ on $J$ in Theorem \ref{extension} exists. If, in addition, the order of $\xi$ is coprime to the class number of $K$, then we obtain a unique idele class character $\chi$ of $K$ solely determined by $\xi$. We summarize this discussion in the corollary below, which will be used later.

\begin{cor}\label{quadratic} Let $K$ be a quadratic extension of $\Q$. Let $\fa$ be a nonzero ideal of $\Z_K$ and $\xi$ a primitive character of $(\Z_K/\fa)^\times$ of order $2^e$ for an integer $e \ge 1$. Suppose one of the following two conditions hold:

(a) $K$ is imaginary with odd class number and $\xi(U_K)=1$;

(b) $K=\Q(\sqrt d)$ for a prime $d \equiv 1 \mod 4$ or $d=2$ is real quadratic, and $\xi(U_K) \subseteq \langle -1 \rangle$.

\noindent Then there is a unique idele class character $\chi$ of $K$ with conductor $\ff_\chi = \fa$ and order $2^e$ such that $\chi^\infty$ on units of $(I_K)^\infty$ lifts $\xi$.
\end{cor}

\begin{proof} The conclusion will follow from Theorem \ref{extension} provided the assumptions there hold. This is obvious for $K$ in (a) and $K = \Q(\sqrt 2)$ in (b). For $K$ in (b), since $d \equiv 1 \mod 4$ is a prime, there is a fundamental unit $\epsilon_K$ of $K$ with $N_{K/\Q}(\epsilon_K) = -1$ (see \cite[Chapter XI, Theorem 3]{Cohn}), hence the class number $h(K)$ of $K$ agrees with the narrow class number of $K$, whose $2$-rank $r(K)$ is one less than the number of prime factors of the discriminant $d_K = d$ of $K$ by Gauss's genus theory (see \cite[Chapter XIII, \S3]{Cohn}).  Therefore $r(K)=1-1 = 0$, implying $h(K)$ is odd.  
Hence by Proposition \ref{constraint-on-xi}, $\xi(-1)$ and $\xi(\epsilon_K)$ uniquely determine $\chi_\infty$ so that all required conditions in Theorem \ref{extension} are satisfied. 
\end{proof}

\subsection{Characters of Galois groups and idele class characters}  
Let $\chi$ be a character of the absolute Galois group $G_K$ of $K$ of finite order. Then the fixed field $F$ of ker $\chi$ is a finite cyclic extension of $K$ with Galois group $\Gal(F/K)$. Let $S$ be the set of finite places of $K$ ramified in $F$. For each finite place $v$ of $K$ outside $S$, there is the associated Frobenius element $\Frob_v$ in $\Gal(F/K)$. The Chebotarev density theorem says that these Frobenius elements are uniformly distributed among the elements in $\Gal(F/K)$. 
By class field theory, the global Artin reciprocity map $\psi_K$ gives rise to an isomorphism from the quotient $I_K/K^\times N_{F/K}(I_F)$ to $\Gal(F/K)$ and $\chi$ on $\Gal(F/K)$ is transported to the idele class character of $K$, also denoted by $\chi$, with kernel $K^\times N_{F/K}(I_F)$ and $\chi_v(\pi_v) = \chi(\Frob_v)$ for all finite $v$ outside $S$. Moreover all finite order idele class characters of $K$ arise this way. Note that at finite $v$ outside $S$, $\mathcal U_v$ is contained in $N_{F/K}(I_F)$ so that $\chi$ is unramified at $v$, while for $v \in S$, $\mathcal U_v$ is not contained in $N_{F/K}(I_F)$ so that $\chi$ is ramified at $v$. Hence $S$ is the support of the conductor $\ff_\chi$ of $\chi$. By strong approximation theorem, the above information on $\chi$ uniquely determines $\chi_\infty$ and $\chi_v$ for $v \in S$.

\medskip

Now we specify $K$ to be a quadratic extension of $\Q$ with $\Gal(K/\Q) = \langle c \rangle$  as in \S2. We  reinterprete the discussion on characters of Galois groups in \S2 in terms of idele class characters. By class field theory, the character $\chi$ of $H$ and its conjugate $\chi^c$ as before correspond to the idele class characters $\chi$ and $\chi^c$ of $K$ with respective conductors $\ff_\chi$ and $\ff_{\chi^c}$, and, as discussed above, their respective restrictions to the group of units in $(I_K)^\infty$ lift primitive characters $\xi=\xi(\chi)$ of $(\Z_K/\ff_\chi)^\times$ and $\xi' = \xi(\chi^c)$ of $(\Z_K/\ff_{\chi^c})^\times$, respectively. The involution $c$ on $K$ swaps the two maximal ideals $v, v^c$ above a prime $p$ splitting in $K$ and induces an isomorphism between $K_v$ and $K_{v^c}$ (both isomorphic to $\Q_p$), while it fixes the unique maximal ideal $v$ above a prime $p$ inert in $K$ and induces the Frobenius automorphism on $F_v$ over $\Q_p$. Consequently, $c$ maps the conductor $\ff_\chi$ to $(\ff_\chi)^c = \ff_{\chi^c}$ so that $\xi = \xi' \circ c$. Since $c$ is an involution, we also have $\xi' = \xi \circ c$. In terms of the Frobenius conjugacy classes in $H$, $c$ maps $\Frob_v$ to $\Frob_{v^c}$.  So at a finite $v$ outside the support of $\ff_\chi \ff_{\chi^c}$, both $\chi$ and $\chi^c$ are unramified, and we have $(\chi^c)_v(\pi_v) = \chi_{v^c}(\pi_{v^c})$.

\subsection{Reduction of order}

In view of Corollary \ref{sufficientcond}, our problem becomes 

\noindent {\bf Question:} Classify finite order idele class characters $\chi$ of $K$ satisfying
\begin{eqnarray}\label{delta}
\chi^c = \chi \cdot \delta
\end{eqnarray}
for some quadratic idele class character $\delta$ of $K$. 

Recall the following well-known fact  (cf. \cite[Proposition 1]{GL}): 

\begin{prop}\label{chi=chic}A finite order idele class character $\mu$ of a quadratic field $K$ satisfies $\mu = \mu^c$ if and only if $\mu = \nu \circ N_{K/\Q}$ for an idele class character $\nu$ of $\Q$. Equivalently, a finite order character $\mu$ of the Galois group $G_K$ extends to a character $\nu$ of $G_{\Q}$ if and only if $\mu = \mu^c$. 
\end{prop}

Such $\mu$ is called the {\it base change} of $\nu$ to $K$.  Since the global Artin reciprocity map $\psi_K$ is the product over places $v$ of $K$ of the local Artin reciprocity map $\psi_{K_v}$ (see, for example \cite[section 6]{Tate}), Proposition \ref{chi=chic} results from  local base change for $\mu_v$, which in turn follows from the functoriality of the local reciprocity map, as explained in \cite[section 2.4]{Serre2}.   Hence if $\chi$ satisfies (\ref{delta}), so does $\chi \cdot \nu \circ N_{K/\Q}$ for all idele class characters $\nu$ of $\Q$. Thus it suffices to study $\chi$ up to multiplication by finite order characters arising from base change. Note that (\ref{delta}) implies $\delta^c = \delta$ so that $\delta$ is from base change.

Clearly $\chi$ and $\chi^c$ have the same order, say, $r$. The condition (\ref{delta}) 
implies $2|r$. Write $r = 2^e r'$, where $e \ge 1$ and $r'$ is odd. Since $2^e$ and $r'$ are coprime, there are integers $a$ and $b$ such that $a2^e + br'=1$. We have $\chi = \chi^{a2^e}\cdot \chi^{br'}$ as a product of two idele class characters $\chi^{a2^e}$ and $\chi^{br'}$ of $K$ of order $r'$ and $2^e$, respectively. Squaring (\ref{delta}) yields 
\begin{eqnarray}\label{delta-squared} 
 (\chi^2)^c = (\chi^c)^2 = \chi^2,
\end{eqnarray} 
implying $(\chi^{a2^e})^c = \chi^{a2^e}$, which comes from base change by Proposition \ref{chi=chic}. This proves

\begin{prop}\label{reduction of order} Let $K/\Q$ be a quadratic extension. Up to multiplication by a character from base change, an idele class character $\chi$ of $K$ of finite order satisfying (\ref{delta}) has order a power of $2$. 
\end{prop}

Therefore {\it we shall assume $\chi$ has order a power of $2$} and study when it satisfies (\ref{delta}). The condition (\ref{delta-squared}) says that its square comes from base change. 
 
\smallskip

If $\chi$ is quadratic and not equal to $\chi^c$, then they differ by a quadratic idele class character of $K$. This proves

\begin{thm}\label{quadratic chi} An idele class character $\chi$ of $K$ of order $2$ satisfies  $\chi^c = \chi \cdot \delta$ for some quadratic idele class character $\delta$ of $K$ if and only if $\chi$ does not arise from base change.
\end{thm}
 
Such $\chi$'s are not faithful on $H$. 

\subsection{Conductors of idele class characters with order a power of $2$}

Before going further, we explore restrictions on the conductor of an idele class character whose order is a power of $2$.

\begin{prop}\label{conductor} Let $K$ be a quadratic extension of $\Q$ with $\Gal(K/\Q) = \langle c \rangle$. Let $\chi$ be an idele class character of $K$ with conductor $\ff$ and order a power of $2$. Let $v$ be a finite place of $K$ such that $\ord_v \ff = m(v) \ge 1$. Suppose $\chi_v$ on $\mathcal U_v$ has order $2^e$. Denote by $\kappa_v$ the residue field at $v$, and by $\pi_v$ a uniformizer of $K_v$. 

(i) If $v$ is above an odd prime $p$, then $m(v)=1$ and $2^e$ divides $|\kappa_v| - 1$. Moreover, $\chi_v^c = \chi_v^p$ on $\mathcal U_v$ has conductor $v$ and the same order as $\chi_v$ if $p$ is inert in $K$; $\chi_v^c = \chi_v$ on $\mathcal U_v$ if $p$ ramifies in $K$; and $\chi_v^c$ is a character on $K_{v^c}^\times$ of conductor $v^c$ and same order as $\chi_v$ if $p$ splits in $K$. 

(ii) If $v$ is above $2$, then $m(v) \ge 2$. Further, more can be said according to the behavior of $2$:

(iia) $2$ splits in $K$. If $m(v) = 2$, then $\mathcal U_v/(1 + \mathcal M_v^{m(v)}) = \langle  -1 \rangle$ and $e=1$. If $m(v) \ge 3$, then 
\[
\mathcal U_v/(1 + \mathcal M_v^{m(v)}) = \langle  -1 \rangle \times \langle  1 + \pi_v^2 \rangle \cong \Z/2\Z \times \Z/2^{m(v)-2}\Z
\]
and $e = m(v)-2$.
 Moreover $\chi_v^c$ is a character of $K_{v^c}^\times$ with the same order as $\chi_v$. 
 
(iib) $2$ is inert in $K$. Choose $\pi_v = 2$. Then $\mathcal U_v = \langle \mu_3 \rangle (1 + \mathcal M_v)$ for a primitive cubic root of unity $\mu_3$ and $\chi_v$ on $\mathcal U_v$ is the trivial extension of $\chi_v$ on $V_v:= (1 + \mathcal M_v)/(1 + \mathcal M_v^{m(v)})$. For $m(v) = 2$, 
\[
V_v = \langle  -1 \rangle\times\langle1+\mu_3 2 \rangle \cong (\Z/2\Z)\times (\Z/2\Z)
\]
 and $e=1$. For $m(v) \ge 3$, 
 \[
 V_v = \langle  -1 \rangle \times \langle 1 + \mu_3 2\rangle \times \langle 1 - \mu_3 2^2\rangle \cong (\Z/2\Z) \times (\Z/2^{m(v)-1}\Z) \times (\Z/2^{m(v)-2}\Z).
 \]
  We have $e = m(v)-1$ if $\chi_v(1 + \mu_3 2)$ has order $2^{m(v)-1}$; otherwise $e = m(v)-2$ and $\chi_v(1 - \mu_3 2^2)$ has order $2^{m(v) - 2}$. Further, $(1 + \mu_3 2)^c = -(1 + \mu_3 2)$ and $(1 - \mu_3 2^2)^c = 1 - \mu_3^2 2^2$. 

(iic) $2$ ramifies in $K$. Then $\mathcal U_v = 1 + \mathcal M_v$ and $V_v:=\mathcal U_v/(1 + \mathcal M_v^{m(v)})$ has order $2^{m(v)-1}$. We have $V_v = \langle1 + \pi_v \rangle$ cyclic for $m(v) = 2$ or $3$; 
\[
V_v = \langle1 + \pi_v \rangle \times \langle 1 + \pi_v^3 \rangle \cong (\Z/4\Z)\times (\Z/2\Z)
\]
 for $m(v) = 4$; and 
 \[
 V_v = \langle1 + \pi_v \rangle \times \langle 1 + \pi_v^3 \rangle \times \langle1 + \pi_v^4 \rangle \cong (\Z/4\Z)\times (\Z/2\Z)\times (\Z/2\Z)
 \]
  for $m(v) = 5$. Moreover, if $e = 1$, then either $m(v) = 2$ with $\chi_v(1 + \pi_v) = -1$, or $m(v) = 4$ with $\chi_v(1 + \pi_v^3) = -1$, or $m(v) = 5$ with $\chi_v(1 + \pi_v^4) = -1$. If $e \ge 2$, we have $m(v) \le 2e+3$. 
  
\end{prop} 

\begin{proof} 
Since $\ord_v \ff = m(v)$, $\chi_v$ on $\mathcal U_v/1 + \mathcal M_v^{m(v)}$ is primitive. 

(i) Assume $v$ is above an odd prime $p$. If $m(v)>1$, $\mathcal U_v/(1 + \mathcal M_v^{m(v)})$ contains the  subgroup $(1 + \mathcal M_v)/(1 + \mathcal M_v^{m(v)})$ whose order is a power of $p$. Since $p$ is coprime to $2^e$, the order of $\chi_v$ on $\mathcal U_v$, so $\chi_v$ is trivial on $(1 + \mathcal M_v)/(1 + \mathcal M_v^{m(v)})$, which implies that $\chi_v$ comes from a character of $\mathcal U_v/(1+\mathcal M_v)$, hence is not primitive on $\mathcal U_v/(1 + \mathcal M_v^{m(v)})$, a contradiction. This proves $m(v)=1$. Note that $\mathcal U_v/(1+\mathcal M_v) \cong \kappa_v^\times$, therefore the order of $\chi_v$ on $\mathcal U_v$, which is $2^e$,  divides $|\kappa_v|-1$. 

If $p$ is inert in $K$, then $\kappa_v$ is a quadratic extension of $\Z/p\Z$ and $c$ induces the Frobenius automorphism on $\kappa_v$. So $\chi_v^c = \chi_v^p$ on $\mathcal U_v$. If $p$ ramifies in $K$, then $\kappa_v \cong \Z/p\Z$ on which $c$ acts trivially, and $\chi_v^c = \chi_v$ on $\mathcal U_v$. If $p$ splits in $K$, then $v$ and $v^c$ are the two places of $K$ above $p$, and $c$ gives rise to the isomorphism $K_v \cong K_{v^c} (\cong \Q_p)$. Thus $\chi_v^c = \chi_v \circ c$ is a character on $K_{v^c}$ with conductor $v^c$ and the same order as $\chi_v$.

(ii) Assume $v$ is above $2$. Let $\pi_v$ be a uniformizer of $K_v$. We distinguish three cases. 

(iia) $2$ splits in $K$. Then $K_v \cong \Q_2$, $\mathcal U_v = 1 + \mathcal M_v$, and $\mathcal U_v/(1 + \mathcal M_v^{m}) \cong (\Z/2^{m}\Z)^\times$, which is $\langle  -1 \rangle$ for $m=2$, and $\langle -1 \rangle \times \langle 1 + \pi_v^2 \rangle$ of order $2^{m-1}$ if $m \ge 3$. So $1 + \mathcal M_v = \langle  -1 \rangle \times (1 + \mathcal M_v^2)$. As $\chi_v$ on $\mathcal U_v$ has order $2^e$, so $\chi_v$ is trivial on $(1 + \mathcal M_v^2)^{2^e}= 1 + \mathcal M_v^{e+2}$.  
Thus  if $m(v) = 2$, then $e=1$ and $\chi_v(-1) = -1$; if $m(v) \ge 3$, then $e = m(v) - 2$ and $\chi_v(1 + \pi_v^2)$ is a primitive $2^e$th root of $1$, while $\chi_v(-1) = \pm 1$. The involution $c$ maps $v$ to $v^c$, we choose $\pi_{v^c}=(\pi_v)^c$, and $\chi_v^c$ is a character on $K_{v^c}^\times$ with conductor equal to $v_c^{m(v)}$ such that $\chi_v^c(-1) = \chi_v(-1)$ and $\chi_v^c(1 + \pi_{v^c}^2) = \chi_v(1 + \pi_v^2)$. 

(iib) $2$ is inert in $K$. Then $K_v$ is a quadratic unramified extension of $\Q_2$ with $2$ as a uniformizer. It contains a primitive cubic root of unity $\mu_3$ such that $\langle \mu_3 \rangle$ represents $\kappa_v^\times$. The involution $c$ induces the Frobenius automorphism on $K_v$, which acts on $\langle \mu_3 \rangle$ by squaring. As $\chi_v$ has order a power $2$, it is trivial on $\langle \mu_3 \rangle$. From $\mathcal U_v = \langle \mu_3 \rangle(1 + \mathcal M_v)$ we deduce that $\chi_v$ on $\mathcal U_v/(1 + \mathcal M_v^{m(v)})$ is a trivial extension of $\chi_v$ on $(1 + \mathcal M_v)/(1 + \mathcal M_v^{m(v)})$, a group of order $4^{m(v)-1}$. 

The group $(1 + \mathcal M_v)/(1 + \mathcal M_v^{2})= \langle -1 \rangle \times \langle 1 + \mu_32 \rangle \cong (\Z/2\Z) \times (\Z/2\Z)$, and  
$(1 + \mathcal M_v)/(1 + \mathcal M_v^{3}) = \langle -1 \rangle \times \langle 1 + \mu_32 \rangle \times\langle 1 - \mu_32^2 \rangle \cong (\Z/2\Z) \times (\Z/4\Z) \times (\Z/2\Z)$ since $(1 + \mu_32)^2 = 1 - 2^2$. One checks inductively on $m\ge 3$ that $(1 + \mathcal M_v)/(1 + \mathcal M_v^{m}) = \langle -1 \rangle \times \langle 1 + \mu_32 \rangle \times\langle 1 - \mu_32^2 \rangle \cong (\Z/2\Z) \times (\Z/2^{m-1}\Z) \times (\Z/2^{m-2}\Z)$. By assumption $\chi_v$ on $\mathcal U_v/(1 + \mathcal M_v^{m(v)})$ has order $2^e$. Then $e=1$ for $m(v)=2$. When $m(v) \ge 3$, either  $\chi_v(1 + \mu_32)$ has order $2^{m(v)-1}$ so that $e= m(v)-1$ or $\chi_v(1 + \mu_32)$ has order $\le 2^{m(v)-2}$ and $\chi_v(1 - \mu_32^2)$ has order $2^{m(v)-2}$ so that $e=m(v)-2$. Note that $(1 + \mu_3 2)^c = 1 + \mu_3^2 2= - (1 + \mu_3 2)$ and $(1 - \mu_32^2)^c = 1 - \mu_3^2 2^2$.  

(iic) If $2$ ramifies in $K$, then $K_v$ is a totally ramified quadratic extension of $\Q_2$ with residue field $\kappa_v \cong \Z/2\Z$, group of units $\mathcal U_v = 1+ \mathcal M_v$ and $2 = u \pi_v^2$ for a unit $u \in \mathcal U_v$. Hence $\mathcal U_v/(1 + \mathcal M_v^{m})$ has order $2^{m-1}$. Note that $(1 + \mathcal M_v^2)^2 = 1 + \mathcal M_v^5$ and $(1+ \pi_v)^4 \in 1 + \mathcal M_v^5$. We find $\mathcal U_v/(1 + \mathcal M_v^{m})$ is $\langle1 + \pi_v \rangle \cong \Z/2^{m-1}\Z$ if $m \le 3$,   
$\langle 1 + \pi_v \rangle \times \langle 1 + \pi_v^3 \rangle \cong (\Z/4\Z) \times (\Z/2\Z)$ if $m = 4$, and 
$\langle 1 + \pi_v \rangle \times \langle 1 + \pi_v^3 \rangle \times \langle1 + \pi_v^4 \rangle \cong (\Z/4\Z) \times (\Z/2\Z) \times (\Z/2\Z)$ if $m = 5$. Further, if $\chi_v$ on $\mathcal U_v$ has order $2$, then there are three possible values for $m(v)$: $m(v) = 2$ with $\chi_v(1 + \pi_v) = -1$ (in which case $\chi_v(-1) = 1$), $m(v) = 4$ with $\chi_v(1+\pi_v^3) = -1$ and $\chi_v(1 + \pi_v) = \pm 1$, and $m(v) = 5$ 
with $\chi_v(1+\pi_v^4) = -1$, $\chi_v(1+\pi_v^3) = \pm 1$ and $\chi_v(1 + \pi_v) = \pm 1$. 

The structure of $\mathcal U_v/(1 + \mathcal M_v^{m})$ for $m \ge 6$ will depend on the field $K$. For example, if the minimal polynomial of $\pi_v$ over $\Q_2$ is $x^2 + 2x + 2$, then $1 + \pi_v$ has order $4$, while if the minimal polynomial is $x^2-2$, then $1 + \pi_v$ has infinite order in $\mathcal U_v$. Since $(1 + \mathcal M_v^s)^2 = 1 + \mathcal M_v^{s+2}$ for $s \ge 3$, if $\chi_v$ on $\mathcal U_v$ has order $2^e$ with $e \ge 2$, we get an upper bound $m(v) \le 2e + 3$. The involution $c$ on $K$ gives the nontrivial automorphism on $K_v$ over $\Q_2$ by sending $\pi_v$ to the other root of its minimal polynomial. 
\end{proof} 

\section{Quadratic idele class characters of quadratic fields up to base change}

In this section we classify quadratic idele class characters of a quadratic field $K$ up to base change and determine the conductors of such characters. We distinguish the discussion according to $K$ being imaginary or real.

\subsection{Classification of quadratic idele class characters over imaginary quadratic fields up to base change}
First consider the case that $K$ is an imaginary quadratic extension of $\Q$ with $\Gal(K/\Q) = \langle c \rangle$. Let $\chi$ be a quadratic idele class character of $K$ with conductor $\ff_\chi$. Then for each $v$ dividing $\ff_\chi$, $\chi_v$ is nontrivial on the group of units $\mathcal U_v$, hence it has order $2$ and thus is trivial on the squares in $\mathcal U_v/(1 + \mathcal M_v^{m(v)})$, where $m(v) = \ord_v \ff_\chi$. Moreover, if $v$ is above an odd prime $p$, then by 
Proposition \ref{conductor} (i), $m(v) = 1$ and $\chi_v$ on $\mathcal U_v/(1 + \mathcal M_v) \cong \kappa_v^\times$ is the unique quadratic character sending the squares in $\kappa_v^\times$ to $1$. If $v$ is above $2$, Proposition \ref{conductor} (ii) says that $m(v) =2$ or $3$ if $2$ is not ramified in $K$, and  $m(v) = 2, 4$, or $5$ otherwise. Further, among these values of $m(v)$, only when $m(v) = 2$ and $2$ is not inert in $K$, the group $\mathcal U_v/(1 + \mathcal M_v^{m(v)})$ is cyclic so that a quadratic character is unique. 

 We proceed to classify such $\chi$ up to multiplication by characters from base change.

\begin{prop}\label{reducing conductor-Kimaginary} 
Let $K= \Q(\sqrt {-d})$ be imaginary quadratic, where $d>0$ is squarefree, and denote  $\Gal(K/\Q)=\langle c \rangle$. Assume the class number of $K$ is odd and  the unit group is $U_K = \{\pm 1\}$ or $U_K = \langle \mu_6 \rangle$, where $\mu_n$ denotes a primitive $n$th root of unity. Then the following hold.

(1) $K= \Q(\sqrt {-d})$ satisfies the assumptions if and only if $d = 2$ or $d$ is a prime $\equiv 3 \mod 4$. Therefore $K/\Q$ is ramified at only one prime. 

(2) Let $v$ be the place of $K$ above an inert odd prime $p$. Then there is a unique quadratic idele class character $\eta$ of $K$ with conductor $\ff_\eta = v$. 


(3) Let $v$ and $v^c$ be the two places of $K$ above a prime $p$ split in $K$. Then there is a unique quadratic idele class character $\eta$ of $K$ with conductor $\ff_\eta$ equal to $vv^c$ for $p$ odd, and $(v v^c)^2$ for $p=2$. Further, for $p=2$, there are two quadratic idele class characters $\eta$ of $K$ with conductor $\ff_\eta = (v v^c)^3$, determined by $\eta_v$ on $\mathcal U_v$. 

(4) Let $v$ be the place of $K$ above $2$ which is either inert or ramified in $K$. Then there is a unique quadratic idele class character $\eta$ of $K$ with conductor $\ff_\eta = v^2$. If $2$ is inert, then there are two quadratic idele class characters $\eta$ of $K$ with conductor  $\ff_\eta = v^3$. 

(5) Let $v$ be the place of $K = \Q(\sqrt {-2})$ above $2$. 
Then there are no quadratic idele class characters $\eta$ of $K$ with conductor  $\ff_\eta = v^4$, and there are two quadratic idele class characters $\eta$ of $K$ with conductor  $\ff_\eta = v^5$, determined by $\eta_v(1 + \pi_v) = \pm 1$, $\eta_v(1+ \pi_v^3) = 1$ and $\eta_v(1 + \pi_v^4) = -1$, where $\pi_v = \sqrt {-2}$. These two characters differ by the quadratic idele class character of $K$ of conductor $v^2$ described in (4).

(6) Let $K = \Q(\sqrt {-d})$ for a prime $d \equiv 3 \mod 4$ and $v$ be the place of $K$ above $d$. Then there are no quadratic idele class characters $\eta$ of $K$ ramified exactly at $v$.

Moreover, the characters $\eta$ in (2)-(4) are from base change, but not (5).
\end{prop}

\begin{proof} For a finite place $v$ of $K$, denote by $\kappa_v$ the residue field of $K_v$.

(1) It follows from Gauss's genus theory (cf. \cite[Chapter 8, Section 3, Theorem 4]{Cohn}) that an imaginary quadratic $K = \Q(\sqrt {-d})$ has odd class number if and only if its discriminant is divisible by only one prime, thus $d = 1, 2$ or a prime $\equiv 3 \mod 4$. The case $d = 1$ is  ruled out because its group of units $\langle \mu_4 \rangle$ is not listed. 

Since we are only concerned with quadratic characters $\eta$ in the remaining statements, when $K = \Q(\sqrt {-3})=\Q(\mu_6)$, it is automatic that $\eta(\mu_3) = 1$. Thus $\eta(U_K) = 1$ if and only if $\eta(-1) = 1$. Hence the argument below for $U_K = \{\pm 1\}$, i.e., $d>3$, also applies to the case $U_K = \langle  \mu_6 \rangle$, i.e., $d=3$.  

(2) If $v$ is above an odd inert prime $p$, then $\kappa_v^\times$ is cyclic of order $p^2-1$ so that $-1$ is a square in $\kappa_v^\times$. Thus the primitive quadratic character $\xi$ on $(\Z_K/v)^\times \cong \kappa_v^\times$ is trivial on $U_K$. By Corollary \ref{quadratic} there is a unique quadratic idele class character $\eta$ of $K$ with conductor $\ff_\eta = v$ such that $\eta^\infty$ lifts $\xi$. As $\xi$ is the unique quadratic character of $\kappa_v^\times$, $\eta$ is the unique quadratic idele class character of $K$ with conductor $v$. Since $v$ is the only place of $K$ above such $p$, $\eta^c$ has the same conductor and the same order as $\eta$, hence $\eta^c = \eta$ arises from base change.

(3) Suppose $v$ and $v^c$ are the two places above a split prime $p$. If $p$ is odd, then $\kappa_v^\times$ and $ \kappa_{v^c}^\times$ are cyclic of order $p-1$. Let $\xi$, $\xi^c$ be the unique quadratic character of $\kappa_v^\times$, $\kappa_{v^c}^\times$, respectively. We have $\xi(-1)= \xi^c(-1) = \pm 1$. Then $\xi$ and $\xi^c$ determine a unique primitive character $\tilde \xi$ on $(\Z_K/v v^c)^\times \cong (\Z_K/v)^\times \times (\Z_K/v^c)^\times$ by Chinese remainder theorem. Note that $\tilde \xi(-1) = \xi(-1)\xi^c(-1)=1$. 
By Corollary \ref{quadratic} there is a unique quadratic idele class character $\eta$ of $K$ with conductor $\ff_\eta = v v^c$ such that $\eta^\infty$ lifts $\tilde \xi$ on $(\Z_K/v v^c)^\times$. Clearly $\eta^c = \eta$ is from base change.

If $p=2$, then by Proposition \ref{conductor} (iia), $\mathcal U_v/(1 + \mathcal M_v^2) = \langle  -1 \rangle$ is cyclic of order $2$, the same argument as above shows the existence of a unique quadratic idele class character $\eta$ of $K$ with conductor $\ff_\eta = (v v^c)^2$. Further, by Proposition \ref{conductor} (iia), $\mathcal U_v/(1 + \mathcal M_v^3) = \langle  -1 \rangle \times \langle 1 + \pi_v^2 \rangle$ is a Klein $4$-group. Let $\eta_v$ be a quadratic character on $\mathcal U_v$ with  conductor $v^3$. Then $\eta_v(1 + \pi_v^2) = -1$ and $\eta_v(-1) = \pm 1$ so there are two possibilities. Notice that if $\eta_v$ extends to a quadratic idele class character $\eta$ of $K$ with conductor $(v v^c)^3$, then the extension is unique. This is because $\eta_{v^c}$ is a quadratic character on $\mathcal U_{v^c}$ with conductor $(v^c)^3$. As such, $\eta_{v^c}(1 + \pi_{v^c}^2) = -1$. Moreover, $\eta_{v^c}(-1)$ must equal $\eta_v(-1)$ in order that $\eta$ is trivial at $-1 \in K^\times$. It remains to prove the existence of an extension $\eta$.

Each $\eta_v$ lifts a primitive character $\xi$ on $(\Z_K/v^3)^\times \cong \mathcal U_v/(1 + \mathcal M_v^3)$. The Galois conjugate $\xi^c$ is a primitive character on $(\Z_K/(v^c)^3)^\times$ satisfying $\xi^c(-1) = \xi(-1)$. By Chinese remainder theorem, $\xi$ and $\xi^c$ determine a unique primitive character $\tilde \xi$ on $(\Z_K/(v v^c)^3)^\times \cong (\Z_K/v^3)^\times \times (\Z_K/(v^c)^3)^\times$. Note that $\tilde \xi(-1) = \xi(-1)\xi^c(-1) = 1$, hence $\tilde \xi(U_K) = 1$.  We conclude from  Corollary \ref{quadratic} the existence of a quadratic idele class character $\eta$ of $K$ extending the given $\eta_v$ on $\mathcal U_v$ with conductor $\ff_\eta = (v v^c)^3$. It follows from the construction that $\eta^c = \eta$ so that $\eta$ is from base change. 

(4) If $v$ is above $2$ which ramifies in $K$, then $\mathcal U_v/(1 + \mathcal M_v^2) = \langle1 + \pi_v \rangle$ is cyclic of order $2$ by Proposition \ref{conductor} (iic). There is a unique quadratic character $\xi$ on $(\Z_K/v^2)^\times \cong \mathcal U_v/(1 + \mathcal M_v^2)$. Hence $\xi^c = \xi$. Moreover $\xi(-1) = 1$ since $-1 \in 1 + \mathcal M_v^2$. By Corollary \ref{quadratic} there is a unique quadratic idele class character $\eta$ of $K$ with conductor $\ff_\eta = v^2$ and $\eta^c = \eta$ is from base change.

Next assume $2$ is inert in $K$. A quadratic idele class character $\eta$ of $K$ with conductor a power of $v$ has to satisfy $\eta_v(-1)=1$ and, as shown in Proposition \ref{conductor} (iib),  possibile conductors for $\eta$ are $v^2$ and $v^3$. We discuss each case. By Proposition \ref{conductor} (iib), $(1 + \mathcal M_v)/(1 + \mathcal M_v^2) = \langle  -1 \rangle \times \langle 1 + \mu_3 2\rangle$ is a Klein $4$-group and $(1 + \mu_3 2)^c = -(1 + \mu_3 2)$. Let $\xi$ be a primitive quadratic character on $(\Z_K/v^2)^\times \cong \mathcal U_v/(1 + \mathcal M_v^2)$ satisfying $\xi(-1) = 1$, then $\xi^c = \xi$. In this case $\xi(1 + \mu_3 2) = -1$ in order to be primitive. By Corollary \ref{quadratic} there is a unique quadratic idele class character $\eta$ of $K$ with conductor $\ff_\eta = v^2$ and $\eta^c = \eta$ is from base change. 

For conductor $v^3$, from Proposition \ref{conductor} (iib) we know $(1+ \mathcal M_v)/(1 + \mathcal M_v^3) = \langle  -1 \rangle \times \langle1 + \mu_3 2 \rangle \times \langle1 - \mu_3 2^2 \rangle \cong (\Z/2\Z) \times (\Z/4\Z) \times (\Z/2\Z)$. Our quadratic character $\eta_v$ on $\mathcal U_v$ with conductor $v^3$ has to satisfy 
$\eta_v(-1) = 1$ and $\eta_v(1 - \mu_3 2^2) = -1$, hence leaving two possibilities $\eta_v(1 + \mu_3 2) = \pm 1$. Note that $\eta_v^c(1 + \mu_3 2) = \eta_v(-(1 + \mu_3 2)) =\eta_v(1 + \mu_3 2)$ and $\eta_v^c(1 - \mu_3 2^2) = \eta_v(1 - \mu_3^2 2^2) = \eta_v(1 + (\mu_3 + 1) 2^2) = \eta_v(1 + 2^2)\eta_v(1 + \mu_3 2^2) = \eta_v(1 - 2^2) \eta_v(1 - \mu_3 2^2) = \eta_v(1 - \mu_3 2^2)$ since $\eta_v(1 - 2^2) = \eta_v(1 + \mu_3 2)^2 = 1$. This shows that both choices of $\eta_v$ on $\mathcal U_v$ satisfy $\eta_v = \eta_v^c$. Now either choice of $\eta_v$ lifts a primitive character $\xi$ on $(\Z_K/v^3)^\times \cong \mathcal U_v/(1 + \mathcal M_v^3)$ satisfying $\xi(-1) = 1$ and $\xi^c = \xi$. By Corollary \ref{quadratic}, either choice of $\eta_v$ extends to a unique quadratic idele class character $\eta$ of $K$ with conductor $\ff_\eta = v^3$ and $\eta^c = \eta$ is from base change. 

(5) Let $v$ be the place of $K = \Q(\sqrt {-2})$ above $2$. Choose $\pi_v = \sqrt {-2}$. Then $(1 + \pi_v)^2 = 1 + \pi_v^2 - \pi_v^3$, $-1 = 1-2 = 1+ \pi_v^2$. By Proposition \ref{conductor} (iic), we know $\mathcal U_v/(1 + \mathcal M_v^4) = \langle 1+\pi_v \rangle \times \langle 1 + \pi_v^3 \rangle$. A quadratic character $\eta_v$ on $\mathcal U_v$ with conductor $v^4$ satisfies $\chi_v(1 + \pi_v^3) = -1$. Then $\chi_v(-1) = \chi_v((1 + \pi_v)^2 (1 + \pi_v^3)) = \chi_v(1 + \pi_v^3) = -1$. Hence there are no quadratic idele class characters $\eta$ of $K$ with conductor $v^4$ since $\eta(-1) = \eta_v(-1) = -1$. 

Finally consider characters of conductor $v^5$. We have  $\mathcal U_v/(1 + \mathcal M_v^5) = \langle 1+\pi_v \rangle \times \langle 1 + \pi_v^3 \rangle \times \langle 1 + \pi_v^4 \rangle$ by Proposition \ref{conductor} (iic). A quadratic character $\eta_v$ on $\mathcal U_v$ with conductor $v^5$ satisfies $\chi_v(1 + \pi_v^4) = -1$. Further, $\chi_v(-1) = \chi_v((1 + \pi_v)^2 (1 + \pi_v^3)) = \chi_v(1 + \pi_v^3)$. In order that $\chi_v$ extends to an idele class character $\eta$ of $K$, we need $\chi_v(-1)= \chi_v(1 + \pi_v^3) = 1$, which leaves two possibilities for $\eta_v$ on $\mathcal U_v$, namely $\eta_v(1+ \pi_v) = \pm 1$. It follows from Corollary \ref{quadratic} that either choice of $\eta_v$ extends to a quadratic idele class character $\eta$ of $K$. To finish, we show that $\eta_v^c \ne \eta_v$. Note that $\pi_v^c = -\pi_v$ so that $(1 + \pi_v)^c = 1 - \pi_v = 1 + \pi_v + \pi_v^3$. Then $\eta_v^c(1 + \pi_v) = \eta_v(1 + \pi_v + \pi_v^3) = \eta_v((1+ \pi_v)(1 + \pi_v^3)(1 + \pi_v^4)) = - \eta_v(1+ \pi_v)$. Therefore $\eta^c \ne \eta$. Clearly these two characters differ by the quadratic idele class character of $K$ of conductor $v^2$ discussed in (4).

(6) Let $\eta$ be a quadratic idele class character of $K$ ramified at $v$. As $d$ is an odd prime, by Proposition \ref{conductor} (1), the conductor of $\eta_v$ is $v$.  Further, since $d \equiv 3 \mod 4$, $-1$ is not a square in $\kappa_v$, hence $\eta_v(-1) = -1$. In order that $\eta(U_K) = 1$, $\eta$ has to ramify at least at two places. Therefore there are no quadratic idele class characters of $K$ which ramify only at $v$. 
\end{proof}

\begin{cor}\label{2nonsplit} Let $K$ be as in Proposition \ref{reducing conductor-Kimaginary}. Let $v$ be a place of $K$ above $2$. Let $\chi$ be a quadratic idele class character of $K$ with conductor $\ff$. Then up to multiplication by a quadratic idele class character of $K$ from base change with conductor supported at $v$, the following hold: 

(i) If $2$ is inert in $K$, then $\ord_v \ff = 0$ or $2$. In the latter case $\chi_v$ satisfies $\chi_v(-1) = -1$ and $\chi_v(1 + \mu_3 2) = 1$ hence is unique;

(ii) If $2$ ramifies in $K$, then $\ord_v \ff = 0, 4$ or $5$. In the latter two cases, $\chi_v$ satisfies $\chi_v(1 + \pi_v) = 1$ so that it is unique if it has conductor $v^4$. 
\end{cor}

\begin{proof} Suppose $\chi_v$ is ramified. 

(i) $2$ is inert in $K$. Then the conductor of $\chi_v$ is either $v^2$ or $v^3$ by Proposition \ref{conductor} (iib). If $\chi_v(-1) = 1$, then the character $\eta$ in Proposition \ref{reducing conductor-Kimaginary} (4) is from base change and $\eta_v = \chi_v$ on the group of units $\mathcal U_v$ so that $v$ is coprime to the conductor of $\chi \cdot \eta^{-1}$.   

Now assume $\chi_v(-1) = -1$. If it has conductor $v^3$, then the proof of Proposition \ref{reducing conductor-Kimaginary} (4) above shows the existence of a quadratic idele class character $\eta$ of $K$ with conductor $v^3$ and is from base change such that $\eta_v(1 - \mu_3 2^2) = \chi_v(1 - \mu_3 2^2)=-1$ and $\eta_v(1 + \mu_3 2) = \chi_v(1 + \mu_3 2) = \pm 1$. In this case $\chi_v \cdot \eta_v^{-1}$  has conductor $v^2$, and $\chi_v \cdot \eta_v^{-1}(-1) = -1$. If 
$\chi_v$ has conductor $v^2$ and $\chi_v(1 + \mu_3 2) = -1$, then the quadratic idele class character $\eta$ of $K$ with conductor $v^2$ as in Proposition \ref{reducing conductor-Kimaginary} (4) is from base change and it is such that $\chi_v \cdot \eta_v^{-1}(1 + \mu_3 2) = 1$ and it has conductor $v^2$. So for the case $\chi_v(-1) = -1$, after multiplying by an idele class character of $K$ arising from base change supported at $v$, we may assume $\ord_v \ff = 2$, $\chi_v(-1) = -1$, and $\chi_v(1 + \mu_3 2) = 1$.  

(ii) $2$ ramifies in $K$. Then the conductor of $\chi_v$ is either $v^2$, $v^4$ or $v^5$ by Proposition \ref{conductor} (iic). If $\chi_v$ has conductor $v^2$, then the character $\eta$ in Proposition \ref{reducing conductor-Kimaginary} (4) with conductor $v^2$ is from base change and $\eta_v = \chi_v$ on the group of units $\mathcal U_v$, therefore $v$ is coprime to the conductor of $\chi \cdot \eta^{-1}$. For the remaining cases, no such $\eta$ with conductor $v^4$ or $v^5$ from base change exists by Proposition \ref{reducing conductor-Kimaginary} (5), hence $\ord_v \ff$ is either $4$ or $5$. Finally, multiplying $\chi$ by the $\eta$ from Proposition \ref{reducing conductor-Kimaginary} (4) with conductor $v^2$ if necessary, we may assume $\chi_v(1 + \pi_v) = 1$ without affecting the conductor of $\chi_v$.  
\end{proof}

Given  a quadratic idele class character $\chi$ of $K$ with conductor $\ff$, Proposition \ref{conductor} describes possible powers of each prime ideal dividing $\ff$. Using Proposition \ref{reducing conductor-Kimaginary}, by multiplying $\chi$ by suitable quadratic idele class characters of $K$ obtained from base change from $\Q$, we can reduce the factors in $\ff$ and limit the places occurring in $\ff$. More precisely, by Proposition \ref{reducing conductor-Kimaginary} (2) we can remove those places above an inert odd prime; for the two places above a split prime, by Proposition \ref{reducing conductor-Kimaginary} (3), we can either remove both of them, or make $\chi$ ramified at one of the designated place; Proposition \ref{reducing conductor-Kimaginary} (4) allows us to simplify the factors above $2$ by either removing or putting further restrictions on $\chi$; while Proposition \ref{reducing conductor-Kimaginary} (5) says that if $v$ is the place above $2$ which ramifies in $K$, then no further reduction at $v$ is possible if $v^4$ or $v^5$ divides $\ff$, and similarly Proposition \ref{reducing conductor-Kimaginary} (6) says that no further reduction at $v$ above the (at most one) ramified odd prime in $K$ is possible. The result at the place above $2$ which does not split in $K$ is summarized in Corollary \ref{2nonsplit}. This proves the first statement of the following classification theorem for quadratic characters.

\begin{thm}\label{2-classification} Let $K$ be an imaginary quadratic extension of $\Q$. Assume the class number of $K$ is odd and $U_K = \{\pm 1\}$ or $\langle \mu_6 \rangle$. For each prime $p$ split in $K$, choose one place $v$ of $K$ above $p$ and let $S_K$ be the collection of these chosen places. Then  

(I) Up to multiplication by characters from base change, a quadratic idele class character $\chi$ of $K$ with conductor $\ff$ satisfies the following conditions:

(A)$_i$ No places above an odd inert prime divide $\ff$;

(B)$_i$ At most one place $v$ above a ramified prime $\equiv 3 \mod 4$ divides $\ff$, and $\ord_v \ff = 1$; 

(C)$_i$ If a place $v$ above a split prime $p$ divides $\ff$ to the power $m(v)$, then $v \in S_K$ and $m(v) = 1$ for $p$ odd, and $m(v) = 2$ or $3$ for $p=2$; 

(D)$_i$ If a place $v$ above $2$ divides $\ff$ to the power $m(v)$ and $2$ does not split in $K$, then $m(v) = 4$ or $5$ and $\chi_v(1 + \pi_v) = 1$ if $2$ ramifies in $K$, and $m(v) = 2$ and $\chi_v(-1) = -1$ and $\chi_v(1 + \mu_3 2) = 1$ if $2$ is inert in $K$. 

(II) Any quadratic idele class character $\chi$ of $K$ with nontrivial conductor $\ff$ satisfying (A)$_i$-(D)$_i$ does not arise from base change. In other words, $\chi^c = \chi \cdot \delta$ for some quadratic idele class character $\delta$ of $K$.

(III) No two distinct quadratic idele class characters of $K$ satisfying (A)$_i$-(D)$_i$ differ by multiplication by an idele class character of $K$ from base change.

(IV) Let $\ff = \prod_{v~{\rm finite}} v^{m(v)}$ be an integral ideal of $K$ with $m(v)$ satisfying (A)$_i$-(D)$_i$. Denote by $r(\ff)$ the number of places $v$ occurring in $\ff$ such that $v$ is above a prime $p \equiv 3 \mod 4$. Then there is a quadratic idele class character $\chi$ of $K$ with conductor $\ff$ satisfying the conditions (A)$_i$-(D)$_i$ if and only if $r(\ff)$ is even if no $v|\ff$ is above $2$, and $r(\ff)$ is odd if there is a prime $v | \ff$ above $2$ with $m(v) = 2$ or $4$. 
\end{thm}
\begin{proof} It remains to prove assertions (II)-(IV).

(II) The second statement follows from Theorem \ref{quadratic chi}. The first statement is equivalent to $\chi \ne \chi^c$ since both are quadratic idele class characters of $K$. 
This is obvious if some place $v$ above a split prime $p$ divides $\ff$ because $v^c$ does not divide $\ff$ by condition (C)$_i$ and it divides the conductor of $\chi^c$. Now suppose no places in $S_K$ divide $\ff$. Then the only places $v$ dividing $\ff$ are above primes $p$ which are either ramified or inert in $K$. Condition (A)$_i$ implies that $p$ cannot be an odd inert prime. If $K = \Q(\sqrt {-2})$, this forces $p=2$ and $\ff = v^4$ or $v^5$ by condition (D)$_i$. We conclude from  Proposition \ref{reducing conductor-Kimaginary} (5) that $\ff = v^5$ and $\chi \ne \chi^c$. The remaining case is $K = \Q(\sqrt {-d})$ for $d>3$ a prime $\equiv 3 \mod 4$ by Proposition \ref{reducing conductor-Kimaginary} (1). So $d$ is the only prime ramified in $K$. If $v$ is above $p=d$, then $\chi_v(-1) = -1$ since $-1$ is not a square in the residue field $\kappa_v$. In order that $\chi(-1) = 1$, one place $v'$ above $2$ has to divide $\ff$. Therefore $2$ is inert in $K$. Condition (D)$_i$ implies that $\chi_{v'}(-1) = -1$, as it should, and $v'$ divides $\ff$ to the power $2$. Using Proposition \ref{conductor} (iib), we get $\chi_{v'}^c(1 + \mu_3 2) = \chi_{v'}(-(1 + \mu_3 2)) = -\chi_{v'}(1 + \mu_3 2)$, showing $\chi^c \ne \chi$. 

(III) Let $\chi$ and $\eta$ be two distinct quadratic idele class characters of $K$ satisfying (A)$_i$-(D)$_i$. Then $\delta :=\chi \eta^{-1} = \chi \eta$ is also a quadratic idele class character of $K$ whose conductor $\ff_\delta$ divides the least common multiple of $\ff_\chi$ and $\ff_\eta$. It clearly satisfies conditions (A)$_i$-(C)$_i$. If (D)$_i$ is also satisfied, then (III) will follow from (II). To check the condition (D), suppose $v$ is a place of $K$ above $2$ dividing $\ff_\chi \ff_\eta$ and $2$ does not split in $K$. We distinguish two cases. 

Case (a) $2$ is inert in $K$. If one of $\chi$ and $\eta$, say, $\eta$, is unramified at $v$, then on $\mathcal U_v$, we have $\delta_v = \chi_v$. If both $\chi$ and $\eta$ are ramified at $v$, then, $\chi_v = \eta_v$ on $\mathcal U_v$ by (D)$_i$ so that $\delta_v$ is unramified at $v$. In both cases the condition (D)$_i$ holds for $\delta_v$.

Case (b) $2$ ramifies in $K$. Then $\delta_v = \chi_v \eta_v$ has conductor at most $v^5$. If it has conductor $v^4$ or $v^5$, then $\delta_v(1 + \pi_v) = 1$ follows from $\chi_v(1 + \pi_v) = \eta_v(1 + \pi_v) = 1$. If $\delta_v$ has conductor less than $v^4$, then it is trivial on $\mathcal U_v$ since $\mathcal U_v/(1 + \mathcal M_v^3) = \langle 1 + \pi_v \rangle$ and $\chi_v(1 + \pi_v) = \eta_v(1 + \pi_v) = 1$. Hence (D)$_i$ holds for $\delta_v$ in both situations.

(IV) First assume $\chi$ exists. Since $K$ is imaginary, the local component $\chi_\infty$ is the trivial character of $\mathbb C^\times$. Therefore $\chi^\infty(-1) = \prod_{v | \ff} \chi_v(-1) = 1$. Suppose $v | \ff$ is above the prime $p$. If $p$ is odd, then $m(v) = 1$ and $\chi_v(-1) = -1$ if and only if $p \equiv 3 \mod 4$. If $p=2$ and $m(v)=2$, then $\chi_v(-1) = -1$ by Proposition \ref{conductor}, (iia) for $2$ splits, and (D)$_i$ for $2$ inert. Further if $p=2$ and $m(v) = 4$ (hence $2$ ramifies), we have $\chi_v(-1) = -1$ as shown in the proof of Proposition \ref{reducing conductor-Kimaginary}, (5). 
Hence $1 = (-1)^{r(\ff)}$ if no prime above $2$ divides $\ff$, and $1 = (-1)^{r(\ff) + 1}$ if there is a prime $v$ above $2$ divides $\ff$ and $m(v) = 2$ or $4$. For the remaining case that some $v|\ff$ with $m(v) = 3$ or $5$ (hence $p=2$), there are two choices of $\chi_v$ satisfying (D)$_i$ so that $\chi_v(-1) = \pm 1$. With given $r(\ff)$ there is a unique choice to make $\chi^\infty(-1) = 1$. This proves the necessity. 

Conversely, assume $\ff$ and $r(\ff)$ satisfy the hypotheses, we proceed to prove the existence of a quadratic idele class character $\chi$ of $K$ as described. For each $v$ occurring in $\ff$, let $\chi_v$ be a quadratic character of $\mathcal U_v$ with conductor $v^{m(v)}$ and satisfying (D)$_i$ if $v$ is above $2$ which does not split in $K$. Such $\chi_v$ is unique except for $v$ above $2$ which either splits in $K$ with $m(v) = 3$ or ramifies in $K$ with $m(v) = 5$. In each case there are two choices for $\chi_v$ and we choose the one so that $\prod_{v|\ff} \chi_v(-1) = 1$. When there is no choice, this identity follows from the assumption on $r(\ff)$.   

Since $(\mathbb Z_K/\ff)^\times \simeq \prod_{v~{\rm finite}}\mathcal U_v/(1 + \mathcal M_v^{m(v)})$, the product $\prod_{v |\ff}\chi_v$ lifts a unique quadratic primitive character $\xi$ of $(\mathbb Z_K/\ff)^\times$ such that $\xi(U_K)=1$. By Corollary \ref{quadratic}, $\xi$ lifts to a quadratic idele class character $\chi$ of $K$ satisfying (A)$_i$-(D)$_i$.
 \end{proof}

The remaining imaginary quadratic $K$ over $\Q$ with odd class number is $K =\Q(\sqrt {-1}) = \Q(\mu_4)$. Its group of units is $U_K=\langle \mu_4 \rangle$. We go over the statements in Proposition \ref{reducing conductor-Kimaginary} one by one. Let $v$ be a place $v$ of $K$ above a prime $p$. For $p$ inert in $K$, i.e. $p \equiv 3 \mod 4$, 
 since $\mu_4$ is a square in $\kappa_v^\times$ which is cyclic of order $p^2-1$, divisible by $8$, statement (2) holds. For $p$ split in $K$, i.e., $p\equiv 1 \mod 4$, the argument in the proof for (3) goes through with $-1$ replaced by $\mu_4$, so (3) also holds. Finally for $p=2$ which ramifies in $K$, choose $\pi_v = \mu_4 - 1$ so that $1 + \pi_v = \mu_4$. Then $\mu_4^c = - \mu_4$ and $\pi_v^c = -\mu_4 - 1 = \mu_4 \pi_v$. By Proposition \ref{conductor} (iic), a quadratic character $\eta_v$ on $\mathcal U_v$ has conductor $v^2, v^4$, or $v^5$. If $\eta_v$ has conductor $v^2$, then $\eta_v(\mu_4) = -1$ and it cannot be extended to an idele class character of $K$ with conductor $v^2$. If $\eta_v$ has conductor $v^4$, then $\eta_v(1 + \pi_v^3) = -1$ and $\eta_v(\mu_4) = 1$ is the only such character which can be extended to a quadratic idele class character $\eta$ of $K$ with conductor $v^4$ by Corollary \ref{quadratic}. In this case $\eta^c = \eta$ so that it is from base change. Similarly, there are two quadratic idele class characters $\eta$ of $K$ with conductor $v^5$, given by $\eta_v(\mu_4) = 1, \eta_v(1 + \pi_v^3) = \pm 1$, and $\eta_v(1 + \pi_v^4) = -1$. We check that $\eta_v^c(1 + \pi_v^3) = \eta_v(1 + \mu_4^3 \pi_v^3) = \eta_v(1 + (1 + \pi_v)^3\pi_v^3) = \eta_v(1 + \pi_v^3 +\pi_v^4)
 =\eta_v((1+ \pi_v^3)(1+\pi_v^4))= - \eta_v(1 + \pi_v^3)$ since $\eta_v$ has conductor $\mathcal M_v^5$. So none of the characters $\eta$ with conductor $v^5$ is from base change. 
 The order of the conductor $\ff$ at $v$ of a quadratic idele class character $\chi$ of $K$ is among $0, 2, 4, 5$.  Multiplying it by the above idele class character $\eta$ of conductor $v^4$ from base change if necessary, we may assume that $\chi$ satisfies $\chi_v(1 + \pi_v^3) = 1$ and  $\ord_v \ff = 5, 2$ or $0$.  
 
 Similar arguments as before with the unit $-1$ replaced by $\mu_4$ prove the following classification theorem for $K = \Q(\sqrt {-1})$. Note that condition (B)$_i$ is automatically satisfied since no odd primes ramify in $\Q(\sqrt {-1})$. 
 
 \begin{thm}\label{2-Q(i)} Let $K = \Q(\sqrt {-1})$. For each prime $p$ split in $K$, choose one place $v$ of $K$ above $p$ and let $S_K$ be the collection of these chosen places. Then    

(I) Up to multiplication by characters from base change, a quadratic idele class character $\chi$ of $K$ with conductor $\ff$ satisfies the conditions (A)$_i$-(C)$_i$ in Theorem \ref{2-classification} and

(D)$_i$' If a place $v$ above $2$ divides $\ff$ to the power $m(v)$, then $m(v) = 2$ or $5$. Further $\chi_v(\mu_4) = -1$ if $m(v)=2$, and $\chi_v(1 + \pi_v^3) = 1$ if $m(v) = 5$.
Here $\pi_v = \mu_4 - 1$.

(IV) Let $\ff = \prod_{v~{\rm finite}} v^{m(v)}$ be an integral ideal of $K$ with $m(v)$ satisfying (A)$_i$-(C)$_i$ and (D)$_i$'. Denote by $r(\ff)$ the number of places $v$ occurring in $\ff$ such that $v$ is above a prime $p \equiv 5 \mod 8$. Then there is a quadratic idele class character $\chi$ of $K$ with conductor $\ff$ satisfying the conditions (A)$_i$-(C)$_i$ and (D)$_i$' if and only if $r(\ff)$ is even if no $v|\ff$ is above $2$, and $r(\ff)$ is odd if there is a prime $v | \ff$ above $2$ with $m(v) = 2$. 

Moreover, (II) and (III) in Theorem \ref{2-classification} hold with (D)$_i$ replaced by (D)$_i$'. 

\end{thm}

\subsection{Classification of quadratic idele class characters over real quadratic fields up to base change}

Next we consider the case for real quadratic $K = \Q(\sqrt d)$ with $d>0$ square-free integer along the same vein. Gauss' genus theory says that the narrow class number $h^+(K)$ of $K$ is odd if and only if its discriminant $d_K$ has one prime factor, which holds precisely when $d$ is a prime $\equiv 1 \mod 4$ or $d=2$. In this case $K$ contains a unit with norm $-1$, hence the class number $h(K) = h^+(K)$. The remaining case for $h(K)$ odd is when $h^+(K) = 2m$ with $m$ odd and there is no unit of negative norm so that $h(K) = h^+(K)/2 = m$ is odd. In this case $d_K$ has two prime factors, hence either $d_K = 4d$ for a prime $d \equiv 3 \mod 4$ or $d = p_1p_2$ with two distinct primes $p_1 \equiv p_2 \mod 4$ since $d_K \equiv 1 \mod 4$. In the former case $d$ a prime, the field $K$ does not contain a unit with norm $-1$, hence $h(K)$ is odd. In the latter case $d = p_1p_2$, if $p_1 \equiv p_2 \equiv 3 \mod 4$, again $K$ has no unit with negative norm so that $h(K)$ is odd. If $p_1 \equiv p_2 \equiv 1 \mod 4$, then the situation is more complicated. Dirichlet (1834) showed that if the Legendre symbol $(\frac{p_1}{p_2})= -1$, then there is a unit of negative norm, so that $h(K) = h^+(K)$ is even. But if $(\frac{p_1}{p_2})= 1$ and the product of the quartic residue symbos $(\frac{p_1}{p_2})_4 (\frac{p_2}{p_1})_4=-1$ then there is no unit of negative norm, so $h(K)$ is odd. 

To illustrate the key points, we consider the simplest case of real quadratic fields and state conclusions similar to the case of imaginary quadratic fields studied above.

\begin{prop}\label{reduction of conductor-Kreal}Let $K=\Q(\sqrt d)$, where $d= 2$ or a prime $\equiv 1 \mod 4$, be real quadratic with Galois group $\Gal(K/\Q)=\langle c \rangle$. Then the following hold.

(1) $d$ is the only prime ramified in $K$. Further, the class number of $K$ is odd and there is a unit in $U_K$ with norm $-1$. 

(2) Let $v$ be the place of $K$ above an inert or ramified odd prime $p$. Then there is a unique quadratic idele class character $\eta$ of $K$ with conductor $v$. It is from base change. 


(3) Let $v$ be a place of $K$ above a prime $p$ split in $K$. Then there is a unique quadratic idele class character $\eta$ of $K$ with conductor equal to $v$ for $p$ odd, and $v^2$ for $p=2$. Further, for $p=2$, there are two quadratic idele class characters $\eta$ of $K$ with conductor $v^3$, determined by $\eta_v(-1) = \pm 1$ and $\eta_v(1 + \pi_v^2) = -1$. None of these characters $\eta$ are from base change, but all $\eta\eta^c$ are.

(4) Let $v$ be the place of $K$ above $2$ which is inert in $K$. Then there are three quadratic idele class character $\eta$ of $K$ with conductor $v^2$. Exactly one of them, satisfying $\eta_v(-1) = 1$ and $\eta_v(1 + \mu_3 2) = - 1$, is from base change. There are four quadratic idele class characters $\eta$ of $K$ with conductor  $v^3$. Exactly two of them, satisfying $\eta_v(-1) = 1$, $\eta_v(1 + \mu_3 2) = \pm 1$ and $\eta_v(1 - \mu_3 2^2)=-1$, are from base change.

(5) Let $v$ be the place of $K = \Q(\sqrt {2})$ above $2$. There is one quadratic idele class character $\eta(2)$ of $K$ of conductor $v^2$, satisfying $\eta_v(1 + \pi_v) = -1$, and it is from base change. 
Up to multiplication by $\eta(2)$, there is one quadratic idele class character $\eta$ of $K$ with conductor $v^4$, satisfying $\eta_v(1 + \pi_v) = 1$ and $\eta_v(1 + \pi_v^3) = -1$, and it is not from base change. Up to multiplication by $\eta(2)$, there are two quadratic idele class characters $\eta$ of $K$ with conductor  $v^5$. 
Exactly one of them, satisfying $\eta_v(1 + \pi_v) = 1$, $\eta_v(1+ \pi_v^3)=-1$ and $\eta_v(1 + \pi_v^4) = -1$, is from base change. Here $\pi_v = \sqrt {2}$. 

\end{prop}

\begin{proof} (1) is already discussed in the paragraph before the proposition. 
In view of Corollary \ref{quadratic} and Theorem \ref{extension}, for our $K$, any quadratic character on the group of units in $I_K^\infty$ has a unique extension to a quadratic idele class character of $I_K$. This makes the proof of the remaining statements simpler than their counterparts in Proposition \ref{reducing conductor-Kimaginary} because there are no constraints for the character on $U_K$. 

(2) For $v$ being the unique place of $K$ above an inert or ramified odd prime $p$, there is a unique quadratic idele class character $\eta$ of $K$ with conductor $v$. Since $c$ stabilizes $v$, we have $\eta = \eta^c$. 

(3) For $v$ a place of $K$ above a split prime $p$, the existence of quadratic idele class characters $\eta$ of $K$ as stated follows from Proposition \ref{conductor} (i) and (iia). As $v^c \ne v$, $\eta$ is not from base change, but the character $\eta \eta^c$ is. 

(4) For $v$ the only place above a split prime $2$, Proposition \ref{conductor} (iib) describes the three (resp. four) quadratic idele class characters of $K$ of conductor $v^2$ (resp. $v^3$). The same argument as in the proof of Proposition \ref{reducing conductor-Kimaginary} (4) shows that $\eta = \eta^c$ if and only if $\eta_v(-1) = 1$ in both cases. 

(5) Choose $\pi_v = \sqrt 2$ so that $\pi_v^c = -\pi_v = \pi_v - \pi_v^3$. By Proposition \ref{conductor} (iic), there is one quadratic idele class character $\eta(2)$ of $K$ of conductor $v^2$ determined by $\eta(2)_v(1 + \pi_v) = -1$. Since $\eta(2)_v^c(1 + \pi_v) = \eta(2)_v(1 - \pi_v) = \eta(2)_v(1 + \pi_v)$, we have $\eta(2) = \eta(2)^c$, hence is from base change. 

 Up to multiplication by $\eta(2)$, for quadratic idele class characters $\eta$ of $K$ of conductor $v^r$ with $r \ge 3$, we may assume $\eta_v(1 + \pi_v) = 1$. By Proposition \ref{conductor} (iic), there is one quadratic character $\eta$ of $K$ with conductor  $v^4$, satisfying $\eta_v(1 + \pi_v^3) = -1$ and $\eta_v(1 + \pi_v) = 1$. We show that $\eta \ne \eta^c$ so that $\eta$ is not a base change from $\Q$. This is because $\eta_v^c(1 + \pi_v^3)= \eta_v(1 - \pi_v^3) = \eta_v(1 + \pi_v^3)$ since $\eta_v$ has conductor $v^4$, and $\eta_v^c(1 +\pi_v) = \eta_v(1-\pi_v) = \eta_v(1 + \pi_v - \pi_v^3) = \eta_v(1 + \pi_v)\eta_v(1 - \pi_v^3) = -\eta_v(1 + \pi_v)$. 

Similarly, by Proposition \ref{conductor} (iic), up to multiplication by $\eta(2)$, there are two quadratic characters $\eta$ of $K$ with conductor  $v^5$, satisfying $\eta_v(1 + \pi_v) = 1$, $\eta_v(1 + \pi_v^4) = -1$ and $\eta(1 + \pi_v^3) = \pm 1$. We check $\eta^c$. Indeed, we have $\eta_v^c(1 + \pi_v^4)= \eta_v(1 + \pi_v^4)$, $\eta_v^c(1 + \pi_v^3)= \eta_v(1 - \pi_v^3) = \eta_v(1 + \pi_v^3)$, and $\eta_v^c(1 +\pi_v) = \eta_v(1-\pi_v) = \eta_v(1 + \pi_v - \pi_v^3) = \eta_v((1 + \pi_v)(1 - \pi_v^3)(1 + \pi_v^4)) = -\eta_v(1 + \pi_v)\eta_v(1 + \pi_v^3)$, which is equal to $\eta_v(1 + \pi_v) = 1$ if and only if $\eta_v(1 + \pi_v^3) = -1$. 
\end{proof}

The conclusion below follows immediately from Proposition \ref{reduction of conductor-Kreal} above by an argument similar to the proof of Corollary \ref{2nonsplit}.

\begin{cor}\label{2nonsplit-Kreal} Let $K$ be as in Proposition \ref{reduction of conductor-Kreal}. Let $v$ be a place of $K$ above $2$. Let $\chi$ be a quadratic idele class character of $K$ with conductor $\ff$. Then up to multiplication by a quadratic idele class character of $K$ from base change with conductor supported at $v$, the following hold: 

(i) If $2$ is inert in $K$, then $\ord_v \ff = 0$ or $2$. In the latter case $\chi_v$ satisfies $\chi_v(-1) = -1$ and $\chi_v(1 + \mu_3 2) = 1$, hence is unique;

(ii) If $2$ ramifies in $K$, then $\ord_v \ff = 0$ or $4$. In the latter case, $\chi_v$ satisfies $\chi_v(1 + \pi_v) = 1$ and $\chi_v(1 + \pi_v^3)=-1$, hence is unique. 
\end{cor}

Now we are ready to state the classification of quadratic idele class characters for real quadratic fields parallel to its counterpart Theorem \ref{2-classification} for imaginary quadratic fields. The proof is similar and hence is omitted.

\begin{thm}\label{2-classification-Kreal} Let $K = \Q(\sqrt d)$ where $d=2$ or a prime $\equiv 1 \mod 4$ be a real quadratic extension of $\Q$. For each prime $p$ split in $K$, choose one place $v$ of $K$ above $p$ and let $S_K$ be the collection of these chosen places. Then  

(I) Up to multiplication by characters from base change, a quadratic idele class character $\chi$ of $K$ with conductor $\ff$ satisfies the following conditions:

(A)$_r$ No places above an odd inert or ramified prime divide $\ff$;


(C)$_r$ If a place $v$ above a split prime $p$ divides $\ff$ to the power $m(v)$, then $v \in S_K$ and $m(v) = 1$ for $p$ odd, and $m(v) = 2$ or $3$ for $p=2$; 

(D)$_r$ If a place $v$ above $2$ divides $\ff$ to the power $m(v)$ and $2$ does not split in $K$, then $m(v) = 4$ and $\chi_v(1 + \pi_v) = 1$ and $\chi_v(1 + \pi_v^3) = -1$ if $2$ ramifies in $K$, and $m(v) = 2$ and $\chi_v(-1) = -1$ and $\chi_v(1 + \mu_3 2) = 1$ if $2$ is inert in $K$. 

(II) Any quadratic idele class character $\chi$ of $K$ with nontrivial conductor $\ff$ satisfying (A)$_r$, (C)$_r$, (D)$_r$ does not arise from base change. In other words, $\chi^c = \chi \cdot \delta$ for some quadratic idele class character $\delta$ of $K$.

(III) No two distinct quadratic idele class characters of $K$ satisfying (A)$_r$, (C)$_r$, (D)$_r$ differ by multiplication by an idele class character of $K$ from base change.

(IV) Given an integral ideal $\ff = \prod_{v~{\rm finite}} v^{m(v)}$ of $K$ with $m(v)$ satisfying (A)$_r$, (C)$_r$ and (D)$_r$, there is a quadratic idele class character $\chi$ of $K$ with conductor $\ff$ so that the conditions (A)$_r$, (C)$_r$ and (D)$_r$ hold. 
\end{thm}
  Statement (IV) above is simpler than its counterpart because in the parallel construction, the character $\xi$ is quadratic so that $\xi(U_K)\subset \{\pm 1\}$ automatically holds and Corollary \ref{quadratic} applies.

\section{Examples of arithmetically equivalent pairs with quadratic characters }\label{sec:Dihedral forms} 
In this section we construct two families of examples of arithmetically equivalent pairs with quadratic characters, which give rise to families of holomorphic weight one cusp forms and Maass cusp forms arising from characters of two different fields, respectively. 

\subsection{Holomorphic weight one cusp forms arising from two different fields}\label{sec:holomorphic forms}

Let $K = \Q(\sqrt{-1})$. Denote by $c$ the complex conjugation so that $\Gal(K/\Q) = \langle c \rangle$. Let $q$ be a prime $\equiv 1 \mod 8$, then $q$ splits in $K$. Let $Q$ be a place of $K$ above $q$ with residue field $\kappa_Q \simeq \mathbb Z_K/Q \simeq \mathbb Z/(q)$. Since $|\kappa_Q^\times| = q-1$ is divisible by $8$, $\sqrt{-1}$ is a square in $\kappa_Q$. The quadratic character $\xi$ of $(\mathbb Z_K/Q)^\times$ is trivial on the group of global units $U_K = \langle \sqrt{-1} \rangle$. By Corollary \ref{quadratic}, there is a unique quadratic idele class character $\chi$ of $K$ with conductor $\ff_\chi = Q$ lifting $\xi$. Then $\chi \ne \chi^c$ and hence is not from base change. 

We take a closer look at $\chi_v(\pi_v)$ for $v \ne Q$. Since $K$ has class number $1$, every maximal ideal $v$ in $\mathbb Z_K$ is principal, that is, $v = (\pi_v)$ for some element $\pi_v \in \mathbb Z_K$. For $v \ne Q$, the value $\chi_v(\pi_v)$ is given by $\xi$ evaluated at $\pi_v \mod Q$. Specifically, the image of a prime number $p \ne q$ in $\mathbb Z[i]/Q$ is the same as its image in $\mathbb Z/(q)$. So $p$ is a quadratic residue in $\kappa_Q$ if and only if $p$ is a quadratic residue in $\mathbb Z/(q)$, i.e., the Legendre symbol $(\frac{p}{q}) = 1$, which holds if and only if $p$ is a quadratic residue in $\kappa_{Q^c}$. 

If $v$ is above $p \equiv 3 \mod 4$, then 
$$\chi_v(\pi_v) = (\chi^c)_v(\pi_v) = (\frac{p}{q})= (\frac{q}{p})$$
 since $q \equiv 1 \mod 8$. If $v$ is above $p \equiv 1 \mod 4$, then $(p) = (\pi_v)(\pi_v)^c = (\pi_v)(\pi_{v^c})$, so $$\chi_v(\pi_v)(\chi^c)_v(\pi_v) = \chi_v(\pi_v)\chi_{v^c}(\pi_{v^c})= (\frac{p}{q})= (\frac{q}{p}).$$ We have shown that $\chi^c = \chi \cdot \delta_{KM/K}$ for $M = \Q(\sqrt q)$. By Theorem \ref{arithequiv}, there is an idele class character $\eta$ of $M$, not self-conjugate, such that $\rho_\chi :=\Ind_K^{\Q} \chi = \Ind_M^{\Q} \eta =: \rho_\eta$ and hence
 $$L(s, \chi, K) = L(s, \eta, M).$$
 
The above computation gives the following description of $L(s, \chi, K)$ for $\Re(s)>1$:

\begin{eqnarray*}
L(s, \chi, K) &=& \frac{1}{1 - \chi_{(1-i)}(\pi_{1-i})2^{-s}} \cdot \frac{1}{1 - \chi_{Q^c}(\pi_{Q^c})q^{-s}} \times \\
&~&\prod_{p \equiv 1 \mod 4,~(\frac{p}{q})=1,~{\rm any}~v|p}\frac{1}{(1 - \chi_v(\pi_v)p^{-s})^2}\times \\
&~&\times \prod_{p \equiv 1 \mod 4, ~(\frac{p}{q})=-1}\frac{1}{1 + (\frac{q}{p})p^{-2s}}\prod_{p \equiv 3 \mod 4}\frac{1}{1 - (\frac{q}{p})p^{-2s}}\\
&=& \sum_{n \ge 1} a_\chi(n) n^{-s}.
\end{eqnarray*}

The representation $\rho_\chi$ has conductor $4q$ (where $4$ comes from the discriminant of $K$ over $\Q$ and $q$ comes from the norm of the conductor of $\chi$). The field $M$ has odd class number and a fundamental unit with norm $-1$. It follows from $q \equiv 1 \mod 8$ that $2$ splits in $M$. In order that  $\rho_\eta$ has conductor $4q$, and $\eta$ not self-conjugate, $\eta$ has conductor $\ff_{\eta} = R^2$ for a place $R$ of $M$ above $2$. By Corollary \ref{quadratic}, $\eta$ is the unique quadratic idele class character of $M$ lifting the unique quadratic character of $(\mathbb Z_M/R^2)^\times$. At $R$, $\eta_R$ takes value $-1$ (resp. $1$) on units congruent to $-1$ (resp. $1$) mod $R^2$. In particular $\eta_R(-1) = -1$, which implies that the two local components of $\eta$ at the two real places of $M$ take opposite values at $-1$. This shows that $\rho_\eta = \rho_\chi$ is odd. Therefore the cusp form $g_\chi=g_\eta$ with $L(s, g_\chi) = L(s, \chi, K)$ is a holomorphic weight $1$ cusp form with level $4q$ and character $(\frac{-4q}{\cdot})$; its Fourier expansion is $$g_\chi(z) = \sum_{n \ge 1}a_\chi(n) e^{2\pi i nz}.$$ 

As $q$ varies, this gives a family of examples of holomorphic weight $1$ cusp forms arising from idele class characters of two different quadratic fields.

\bigskip
\subsection{Maass cusp forms arising from two different fields}

Consider two real quadratic fields $K=\mathbb Q(\sqrt t)$ and $M = \mathbb Q(\sqrt q)$, where  $t $ and $q$ are two distinct primes $\equiv 1 \mod 4$ such that the Legendre symbol $(\frac{q}{t}) = (\frac{t}{q}) = 1$. Then $K$ has odd class number, and it has a fundamental unit $\epsilon_K$ with norm $-1$. The same holds for $M$. Denote by $\sigma$ the generator of the Galois group $\Gal(K/\mathbb Q)$ and $\tau$ that of $\Gal(M/\Q)$. By choice, $q$ splits in $K$, that is, $(q) = Q Q^{\sigma}$ in $\mathbb Z_K$ and $t$ splits in $M$, namely $(t) = T T^\tau$ in $\mathbb Z_M$.  

By Corollary \ref{quadratic} there is a unique quadratic idele class character $\chi$ of $K$ with conductor $Q$. It lifts the quadratic character of $(\mathbb Z_K/Q)^\times \cong (\mathbb Z/q\mathbb Z)^\times$. Further, since $\chi_Q(-1)=1$, the local components of $\chi$ at the two infinite places $\infty_1$ and $\infty_2$ of $K$ agree, hence the induced representation $\rho_\chi :=\Ind_{G_K}^{G_{\Q}} \chi$ is even. In fact, at the complex conjugation $c$ in $G_\Q$, we have $\rho_\chi(c) = \pm Id$ with the sign given by the value $\chi_{\infty_1}(-1)=\chi_{\infty_2}(-1)=\chi_Q(\epsilon_K)$, where $\epsilon_K$ is a fundamental unit of $K$ with norm $-1$. Therefore, the sign is $+$ if and only if $\epsilon_K$ is a square in the residue field $\mathbb Z_K/Q$ at $Q$.

We compare $\chi$ and its conjugate $\chi^{\sigma}$. 

\begin{proposition} $\chi \chi^{\sigma} = \delta_{KM/K}$.
\end{proposition}

\begin{proof} Denote by $h(K)$ the class number of $K$. Let $v$ be a place of $K$ above a prime $p$. If $p$ splits in $K$, that is, $(p) = v v^{\sigma}$ in $\mathbb Z_K$, we have $(p)^{h(K)} = v^{h(K)}(v^\sigma)^{h(K)} = (\beta_v)(\beta_{v^{\sigma}})$ and $(\beta_{v^{\sigma}})= ((\beta_v)^{\sigma})$. Recall from the  definition of $\chi$ in the proof of Theorem \ref{extension} that, for $p \ne q$, we have
\begin{eqnarray*}
\chi_v(\pi_v)\chi_v^{\sigma}(\pi_v)&=& \chi_v(\beta_v)\chi_v^{\sigma}(\beta_v)=\chi_v(\beta_v)\chi_v((\beta_v)^{\sigma})\\
& =& \chi_v(p)^{h(K)} = \chi_v(p) = (\frac{q}{p}) = \delta_{M/\mathbb Q}(p)=\delta_{KM/K}(\pi_v)
\end{eqnarray*} since $p$ splits completely in $K$ and is unramified in $M$, so it splits completely in $KM$ (i.e. $v$ splits in $KM$) if and only if $p$ splits in $M$. If $v$ is inert in $K$, then $v=(p)$ and
$$\chi_v(\pi_v)\chi_v^{\sigma}(\pi_v)=\chi_v(p)^2=1=\delta_{KM/K}(\pi_v)$$ since such $v$ is unramified in $KM$ and its residue field already has $p^2$ elements, it has to split in $KM$ or else $KM$ would have a place with $p^4$ elements in its residue field, which is impossible because $Gal(KM/\mathbb Q)$ is a Klein $4$ group. Now both $\chi \chi^{\sigma}$ and $\delta_{KM/K}$ are idele class characters of $K$ which agree at all but finitely many places, they agree. 
\end{proof}

It then follows from the proof of Theorem \ref{arithequiv} that there is an idele class character $\eta$ of $M$ such that $L(s, \chi, K) = L(s, \eta, M)$. This in turn implies that $\eta$ is quadratic and ramified exactly at one of the places above $t$, say, $T$, and it has conductor $T$. Thus $\eta$ is unique by Corollary \ref{quadratic} and it is not self-conjugate. 

Let $g = g_\chi = g_\eta$ be the Maass cusp form with associated L-function $L(s, g) = L(s, \chi, K) = L(s, \eta, M) = \sum_{n \ge 1}a_\chi(n)n^{-s}$. There are two possible explicit Fourier expansions for $g(z)$, depending on the sign of $\rho_\chi(c) = \pm Id$ given by $\chi_Q(\epsilon_K)$:   
\[
g_\chi(z) = \sum_{\mathfrak a} \chi(\mathfrak a) \sqrt{y}K_0(2\pi N(\mathfrak a) y) 2 \begin{cases} \cos(2\pi N(\mathfrak a) x),& {\rm if~} \rho_\chi(c) = Id; \\ \sin(2\pi N(\mathfrak a)x),& {\rm if~} \rho_\chi(c) = -Id ,\end{cases} 
\]
 where  $K_0$ is the $K$-Bessel function (see \cite[Theorem 1.9.1]{Bump}). 
We see from the two examples below that both signs can occur.
\bigskip

\noindent Example 1. $t=5$ and $q = 29$ so that $K = \Q(\sqrt 5)$, $M = \Q(\sqrt {29})$ and $(\frac{29}{5})=1$. We know that $\mathbb Z_K$ has class number $1$ (in fact, it is a Euclidean domain) and  
\[
\epsilon_K = \frac {1+\sqrt{5}}2
\]
is a fundamental unit of norm $-1$. The ideal $(29)$ factors as $Q\cdot Q^\sigma$ with 
\[
Q =(7+2\sqrt{5}).
\]
To compute $\chi_{Q}(\epsilon_K)$, we use Euler's criterion: If $Q$ is a prime ideal and $\alpha\in \mathbb Z_K$ is coprime to $Q$, then
\[
\chi_Q(\alpha) \equiv \alpha^{(N(Q)-1)/2} \bmod Q.
\]
As our $Q$ has norm  $N(Q)=29$, we compute by using the Euclidean division algorithm in $\mathbb Z_K$
\[
\epsilon_K^{14} = \frac{843 + 377 \sqrt{5}}2 = (7+2\sqrt{5}) \cdot \Big(20+33 \frac {1+\sqrt{5}}2\Big) +1
\]
so that
\[
\epsilon_K^{(29-1)/2} = +1\bmod Q
\]
which shows that 
\[
\chi_{Q}(\epsilon_K) = +1.
\]
Therefore $\rho_\chi(c) = Id$ and the Maass form $g_\chi$ has Fourier expansion
\[
g_\chi(x+iy) = \sum_{n\ge 1} a_\chi(n) \sqrt{y}K_0(2\pi n y) 2\cos(2\pi nx),
\]
 where $a_\chi(n) = \sum_{ N(\fa) = n} \chi(\fa)$ is the coefficient of $n^{-s}$ in $L(s, \chi, K)$. 

\noindent Example 2. $t=5$ and $q = 41$ so that $K = \Q(\sqrt 5)$, $M = \Q(\sqrt {41})$ and $(\frac{41}{5})=1$. The ideal $(41)$ factors as $Q\cdot Q^{\sigma}$ with
\[
Q =(6+\frac {1+\sqrt{5}}2)
\] and $N(Q)=41$ in this case. Then
\[
\epsilon_K^{(41-1)/2} = \frac{1}{2} \left(15127+6765 \sqrt{5}\right) = (6+\frac {1+\sqrt{5}}2) \cdot (549 +888 \frac {1+\sqrt{5}}2) -1
\]
so that 
\[
\chi_{Q}(\epsilon_K) \equiv \epsilon_K^{(41-1)/2} = -1\bmod Q.
\] Therefore $\rho_\chi(c) = -Id$ and $g_\chi$ is a Maass form with Fourier expansion
\[
g_\chi(x+iy) = \sum_{n\ge 1} a_\chi(n) \sqrt{y}K_0(2\pi n y) 2\sin(2\pi nx).
\]


%
%



\end{document}